\documentclass[11pt,letterpaper]{amsart}
\usepackage{amssymb,
                       amsmath,
                               txfonts,
                                      mathrsfs,
                                              titletoc,
                                                    appendix,
                                                    relsize}
\usepackage{graphicx}
\usepackage{xcolor}
\usepackage{subcaption}

\usepackage[alphabetic]{amsrefs}

 \newtheorem{thm}{Theorem}[section]
 \newtheorem{cor}[thm]{Corollary}
   
  \newtheorem{cjt*}{Conjecture}
 \newtheorem{lem}[thm]{Lemma}

 \newtheorem{defn}[thm]{Definition}
 \newtheorem{rem}[thm]{Remark}
 \numberwithin{equation}{section}
\newtheorem{lem*}{Lemma}
\newtheorem{cor*}{Corollary}

\newenvironment{pfT2}{\medskip \noindent
{\it Proof of Theorem \ref{T2}.}}{\hfill $\square$\par
} 

\newenvironment{pfT1}{\medskip \noindent
{\it Proof of Theorem \ref{T1}.}}{\hfill $\square$\par
}

\newenvironment{pf}{\medskip \noindent
{\bf Proof of Theorem \ref{mainthm1}.}}{\hfill \rule{.5em}{1em}
\\}

\newenvironment{rmk}{\mbox{ }\\{\bf  Remark}\mbox{ }}{
\hfill $\Box$\mbox{}\bigskip}

\def\p#1{\partial #1}

\def\la{\lambda}
\def\La{\Lambda}

\def\th{\theta}

\def\R{\Bbb{R}}

\begin{document}

\title[Detect duality obstruction of calibrations]{Detect duality obstruction of calibrations \\in smooth category}%
\author{Yongsheng Zhang}
\address{Academy for Multidisciplinary Studies, Capital Normal University, Beijing 100048, P. R. China}
\email{yongsheng.chang@gmail.com}
\date{\today}
\keywords{Calibration, comass control, duality obstruction, realization problem} 
 \subjclass{Primary~53C38}
 %

\begin{abstract}
This paper consists of three parts:
(a) exhibit a new gluing result which can dramatically simplify extensions of calibration pairs;
(b) observe that every Lawlor cone can support coflat calibrations singular only at the origin;
(c)  show that there exist many Lawlor cones which cannot support any smooth calibrations.
As an application,  we extend our previous work on detecting duality obstruction of calibrations in the smooth category.
\end{abstract}
\maketitle
\section{Introduction}\label{S1} 
           In \cite{Z12, z} the author developed a very ``soft" framework about extending calibration pairs
           which play effective roles in detecting duality obstruction and realization problem about  singularities (to be explained below).
           By combining these methods  the framework has been carried forward to attack several  long-term standing conjectures regarding area-minimizing currents, e.g. \cite{L1, L2, L3}.

             The reason for the calibration method being so useful is due to the fundamental theorem of calibrated geometry. 
          Detailed explanations about  the theorem will be given in \S 2.
            A special situation of   the theorem 
            says that 
            any calibrated closed smooth submanifold is mass-minimizing in its homology class.
         So, the {\bf duality question} in the smooth category is the following:
         \begin{quote}
         $(\star)$ If a closed smooth submanifold $M$ in some Riemannian manifold $(X, g)$ is already mass-minimizing in its homology class,
         then can we have some smooth calibration form to calibrate $M$?
         \end{quote}
         
               In  Example 3 of \cite{Z12}
          we observe  some intriguing duality obstruction phenomenon for the first time.
        It leads us to some smooth closed hypersurface $M$ of dimension $m\geq 7$ in some Riemannian manifold $(X,g)$ of dimension $m+1$
          which minimizes mass in its $\R$-homology class $[M]$
          but cannot be calibrated by any smooth calibration form.
        Also in \cite{Z12},  
                   examples of high codimension with the duality obstruction
                       can be gained.
                       Among other constructions, 
                         \begin{quote}
          $(\ast)$
                for simplicity,   consider the Cartesian product  $(X\times \mathbb S^\ell, g\oplus g_0)$ 
        where  $(\mathbb S^\ell, g_0)$ is a standard unit sphere  with $\ell\geq 1$.
         Take $q\in (\mathbb S^\ell, g_0)$.
         Then
 the submanifold $M\times \{q\}$ 
          gives us a homologically mass-minimizing smooth submanifold of codimension $\ell+1$ 
          and it cannot be calibrated by any smooth calibration form.
           \end{quote}
          
          This obstruction in $(\ast)$ comes from some delicate balance between geometry and topology.
          And it has been further pointed out on page 431 of \cite{FH} that our obstruction even exists in local around the submanifold in the category of continuous calibration forms.
          To distinguish this local obstruction $(\ast)$ from global ones, 
          the author  used the following example to explain to others  since 2020
      and  such kind of illustration can trace back  at least  to \cite{FH}.
          Let $Y=S^1\times S^{2020}$ and choose a closed embedded curve $\gamma_2$ to be a representative of twice the generator ${\tt e}$ of $H_1(Y; \mathbb Z)$.
          Then as in \cite{Z12} we can conformally change background metric such that $\gamma_2$ becomes  the unique length minimizer in its homology class.
          As a result, any length minimizer $\gamma_1$ in class ${\tt e}$ cannot be calibrated by any smooth calibration
          since otherwise the integral current $2\gamma_1$ is also calibrated and thus length minimizing in $[\gamma_2]$ as well which contradicts with the uniqueness in the above.
               Here the obstruction comes from global in the sense that one cannot apply the idea in a sufficiently small neighborhood of $\gamma_2$ for the duality obstruction phenomenon.


               In this paper, we extend our (localizable) duality obstruction for question $(\star)$ 
                                     to genuine high codimension situations beyond that in $(\ast)$.
                                     The idea is to find minimizing cones which support coflat calibrations singular only at the origin but cannot be calibrated by any smooth calibration form,
                                     and then apply our extension result for calibration pairs in \cite{Z12} (Theorem \ref{conecal} in below).
                                     
                                     To be able to include more general local models which we can start with, 
                                     we establish the following two results.
                                     
                                     \begin{thm}\label{T1}
                                     Let $L$ be a closed embedded submanifold in $\mathbb S^N\subset \mathbb R^{N+1}$.
                                     If the cone $C(L)=\{tp\, : \, t\in(0,\infty),\, p\in L\}$ can be shown area-minimizing by Lawlor's criterion,
                                     then $C(L)$ supports  coflat calibrations singular only at the origin.
                                     \end{thm}
                                     
                                     \begin{rem}
                            When a closed submanifold $L\subset \mathbb S^N$ is not totally geodesic,
                              the cone $C(L)$ has only one singular point at the origin
                              and it is called a regular cone.
                                     \end{rem}
                                  \begin{rem}
                                    We say a current supports a calibration if it is calibrated by the calibration, see Definition \ref{calibratable}.
                                     \end{rem}

                                     \begin{thm}\label{T2}
                                     Let $L_i$ be a closed embedded  minimal submanifold in $\mathbb S^{N_i}$ for $i=1,\cdots, n$.
                                     Suppose that $L_1$ is of codimension one. 
                                     Then the cone $C(L_1\times\cdots\times L_n)$ over their minimal product $L_1\times\cdots\times L_n$ supports no smooth calibration.
                                     \end{thm}

                             In conjunction with recent configuration result for area-minimizing cones in \cite{Z25},
                             these two theorems in particular imply the following.
                                     
                             \begin{cor}\label{C1}
                             Given any closed embedded minimal hypersurface $L$ in $\mathbb S^N\subset \mathbb R^{N+1}$,
                             let $L^{\times n}$ denote the minimal product of $n$ copies of $L$.
                             Then, when $n$ is sufficiently large,  the cone $C(L^{\times n})$ over  $L^{\times n}$ supports coflat calibrations singular only at the origin but no smooth calibration forms.
                             \end{cor}

                         The paper is organized as follows.
                          To be self-contained,
                          we review some basics and useful tools from calibrated geometry
                           in \S \ref{S2}
                          and moreover establish a much better gluing result in \S \ref{S3} which can significantly simplify/improve the gluing procedures in the  framework.
                         A proof of Theorem \ref{T1} will be included in \S \ref{S4},
                         and
                     a shortcut for the realization result in \cite{z} will be found.
                         In \S \ref{S5} we give a proof of Theorem \ref{T2} by virtue of a result by J. Simons. 
                          In \S \ref{S6} we  extend our duality obstruction beyond $(\ast)$
                          and situation with more complicated singular set will also be mentioned.
                          
                          {\ }
                          

                            \section{Basics of calibrations}\label{S2}
Let us first review some necessary definitions and basic properties in the theory of calibrations.

                              \subsection{Definitions and basic properties. }
                              Let us begin with smooth forms.

  \begin{defn}\label{comassdef}
            Let $\phi$ be a smooth $m$-form on a Riemannian manifold $(X,g)$.
            At a point $x\in X$ the \textbf{comass} of $\phi_x$ is defined to be
                     \begin{equation*}
                     \|\phi\|_{x,g}^*=\max \ \{\phi_x( \overrightarrow V_x) : \overrightarrow V_x \ \text{is a unit simple m-vector at x}\}.
                     \end{equation*}
             Here {``simple"} means $\overrightarrow V_x=e_1\wedge e_2\cdots \wedge e_m$
             for some $e_i\in T_xX$.
             \end{defn}
       
            A remark in \cite{Z12} about the comass that we shall use is the following.
            \begin{lem}\label{comassf}
            At a point $x$ where $\phi_x\neq0$, we have
                       \begin{equation*}
                            \begin{split}
                            \|\phi\|^*_{x,g}& =  \max \{\phi(\overrightarrow V_x): \overrightarrow V_x\ \text{is a simple $m$-vector  at } x \text{ with}\ 
\|\overrightarrow V_x\|_{g}=1\}\\
                  &=\max\{1/{\|\overrightarrow V_x\|_{g}}: \overrightarrow V_x\ \text{is\ a\ simple\ $m$-vector\ at } x \text{ with}\ \phi(\overrightarrow V_x)=1\}\ \\
                  &=1/\min\{\|\overrightarrow V_x\|_{g}:\overrightarrow V_x\ \text{is\ a\ simple\ $m$-vector\ at } x \text{ with}\ \phi(\overrightarrow V_x)=1\}.
                             \end{split}
                         \end{equation*}
             \end{lem}
             
             The dual complex of smooth forms gives us currents.
             To be more precise,  by $(\mathscr E'_*(X),d)$ we mean the dual complex of the {\it de Rham} complex of $X$.
               Elements of $\mathscr E'_m(X)$ are $m$-dimensional {\it de Rham} 
                                                                                   {currents} 
                                                                                   (with compact support)
               and $d$ is the adjoint of exterior differentiation.
             
               \begin{defn}\label{dmass}
               The \textbf{mass} $\mathrm{\mathbf{M}}(T)$ of $T\in\mathscr E'_m(X)$ is defined to be
               $$\sup \{T(\psi):\psi\ smooth\ m\text{-}form\ with\  \sup_{x\in X}\|\psi\|_{x, g}^*\leq 1\}.$$
               \end{defn}
          When $T$ has finite mass, 
          there exists a unique {Radon} measure $\|T\|$ satisfying
                 $$ \int_X f\cdot d\|T\|=\sup\{T(\psi): \|\psi\|_{x,g}^*\leq f(x) \}$$
          for any nonnegative continuous function $f$ on $X$.
          Moreover, the {Radon}-{Nikodym} Theorem asserts the existence of
          a $\|T\|$ measurable tangent $m$-vector field $\overrightarrow T$ a.e. with 
          vectors $\overrightarrow T_x \in \Lambda^m T_xX$ of unit length in the dual norm of the comass norm, with
                \begin{equation}\label{current}
                T(\psi)= \int_X\psi_x(\overrightarrow {T_x})\ d \|T\|(x)\ \  \text{for any smooth $m$-form }\psi,
                \end{equation}
           \text{or\ briefly}
           $T = \overrightarrow T\cdot \|T\|\ a.e.\ \|T\|.$
                          
                          The groundbreaking  paper \cite{FF}  by Federer and Fleming  introduced the concepts  of normal currents and integral currents. 
                 Let $\mathbb M_m(X)=\{T\in\mathscr E'_m(X): \mathrm{\mathbf{M}}(T)<\infty\}$.
                 Then $N_m(X)=\{T\in\mathbb M_m(X): dT\in\mathbb M_{m-1}(X)\}$ is the space of $m$-dimensional 
                                                         {normal\ currents}.
                 Note that $(N_*(X),d)$ form a chain complex
                 and moreover
                 by {de Rham}, Federer and Fleming, we have
                 natural isomorphisms
                            \begin{equation}\label{iso}
                            H_*(\mathscr E'_*(X))\cong H_*(X;\mathbb R)\cong H_*(N_*(X)).
                            \end{equation}
 
In this paper we adopt the following terminologies.
  
           \begin{defn}\label{calibration}
           A smooth form $\phi$ on $(X,g)$ is called a \textbf{calibration} if 
           $\sup_{X}\|\phi\|_{g}^*= 1$ 
           and
           $d\phi=0.$
           Such a triple $(X,\phi,g)$ is called a \textbf{calibrated manifold}.
          \end{defn}
          
          \begin{defn}\label{calibratable}
                   Let $\phi$ be a calibration on $(X,g)$.
                   We say that a current $T$ of local finite mass is \textbf{calibrated} by $\phi$, if 
                        $\phi_x(\overrightarrow T_x)=1\ a.a.\ x\in X\ \text{for}\ \|T\|.$
                     In this case we say $T$ \textbf{supports} calibration $\phi$.
          \end{defn}

                    The  fundamental theorem of calibrated geometry below
                    says that every calibrated current is mass-minimizing in its homology class.
                    By default  the homology is taken for normal currents 
                    but in fact it holds for competitors of
                    {\it de Rham} currents.
                              \begin{thm}[\cite{HL}]\label{hl}
                                           If $T$ is a calibrated current 
                                with compact support
                                           in $(X,\phi,g)$ and
                                           $T'$ is any compactly supported current homologous to $T$(i.e., $T-T'$ is a boundary and in particular $dT=dT'$),
                                           then
                                                          \begin{equation*}
                                                          \mathrm{\mathbf{M}}(T)\leq  \mathrm{\mathbf{M}}(T')
                                                          \end{equation*}
                                          with equality if and only if $T'$ is calibrated as well.
                               \end{thm}
                               
                                                       The subclass $\mathscr R_m(X)\subset \mathbb M_m(X)$ of rectifiable $m$-currents
                            is 
                             the closure of the integral Lipschitz $m$-chains in the $\bold M$-topology. 
                            Then $I_m(X):=\big\{T\in \mathscr R_m(X): dT\in \mathscr R_{m-1}(X)\big\}$
                               is the set of {integral} $m$-currents.
                               Federer and Fleming \cite{FF} also established the corresponding isomorphism 
                               \begin{equation}\label{isoz}
                               H_*(I_*(X))\cong H_*(X;\mathbb Z).
                               \end{equation}

                               So a calibrated integral current
                                minimizes the mass not only in the homology class of integral currents but also that for normal currents.
                               If one merely focuses on its minimality in \eqref{isoz} then it is called (homologically) area-minimizing.

                               In practice,  sometimes one cannot avoid calibrations with singularities.
                               
                                \begin{defn}\label{coflat}
           Let $\phi$ be a calibration of degree $m$ on $X-S_\phi$,
           where $S_\phi$ is a closed subset of $X$ of Hausdorff $m$-measure zero.
           Then $\phi$ is called a \textbf{coflat calibration} on $X$.
           We say $\phi$ \textbf{calibrates} a current, if it is calibrated by $\phi$ on $X-S_\phi$.
           Similarly, in such a case we say the current \textbf{supports} the coflat calibration $\phi$.
          \end{defn}

              A coflat version (Theorem 4.9 in \cite{HL}) of the fundamental theorem of calibrated geometry is also true.
              Namely,    a current calibrated by a coflat calibration $\phi$ in $X-S_\phi$ is homologically mass-minimizing in the sense of Theorem \ref{hl}.
              More generally,
                     the same property holds even for certain non-continuous ``calibrations" with larger dimensional singular set, e.g. see Lawlor's calibrations in \S4, 
                     which provide a sufficient criterion for a regular cone to be  area-minimizing.


      Although our methods in \cite{Z12, z} are quite soft, 
           however there exists certain local rigidity 
           in the sense that after blowing-up 
                   the local models always live in some Euclidean spaces with standard metrics.
           Hence, to begin with suitable local models is important to us.
           
         According to a fundamental result (Theorem 5.4.3 in \cite{F})
           there always exists at least one tangent cone for area-minimizing integral current,
           and moreover, the tangent cone is area-minimizing.
           In this paper we mainly focus on the situation of regular area-minimizing cones in  standard Euclidean space $(\R^n, g_E)$.

      \subsection{Four lemmas. }        
             A beautiful canonical decomposition lemma for an $m$-form with respect to certain $m$-plane in $\R^n$
            is the following (its proof in \cite{HL1} inspires a remark for $m=1$ and Lemma \ref{L2}, so we include it in full details).
            
            \begin{lem}[Lemma 2.12 in \cite{HL1}] \label{hl1}
             Let $\xi\in \Lambda^m \mathbb{R}^n$ be a simple $m$-vector with
             $V = span\ \xi$. Suppose $\phi\in\Lambda^m\big(\mathbb{R}^n\big)^*$ satisfies  $\phi(\xi)= 1$. 
              Then there exists a unique oriented complementary subspace $W$ to $V$  with the following property.
                      For any basis $v_1, \cdots, v_n$ of $\mathbb{R}^n$ such that $\xi=v_1\wedge... \wedge v_m$
                      and $v_{m+1}, \cdots, v_n$ is basis for $W$, 
                      one has that
                                 \begin{equation}
                                 \phi=v_1^*\wedge \cdots \wedge v_m^*+ \sum a_Iv_I^*,
                                  \end{equation}
                                   where $a_I=0$ whenever $i_{m-1}\leq m$. 
                                   Here $I=\{i_1,\cdots, i_m\}$ with $i_1<\cdots<i_m$.
              \end{lem}
              \begin{rem}
              Note that the decomposition 
              does not depends on ambient metric. 
              \end{rem}
                 \begin{proof}
                      Define linear map
                      $$
                      \la:  \Lambda^{m-1} V \longrightarrow  \text{Hom} ( \R^n, \, \R)
                      $$
                      by 
                      $
                      \eta \longmapsto \eta \lrcorner \phi ,
                      $
                      namely, $\la(\eta)(v)=\phi(\eta\wedge v)$.
                      Set 
                                      $\eta_i=(-1)^{m+i}       v_1\wedge\cdots \wedge \hat v_i\wedge \cdots \wedge v_m$
                                      for $i=1,\cdots, m$.
                                      Consider forms
                                      $\la_i=\la(\eta_i)$.
                                      Then we have
                          \begin{equation}\label{pair}
                                    \la_i(v_j)=(-1)^{m+i}
                                              \phi
                                              (  v_1\wedge\cdots \wedge \hat v_i\wedge \cdots \wedge v_m \wedge v_j)=\delta_{ij}
                         \end{equation}
                                      for $1\leq i\leq m$ and $1\leq j\leq m$.
                                      Since
                                         $(\la_1\wedge \cdots \wedge\la_m)(\xi)=1$,
                                         it follows that
                                 \begin{equation}\label{W}
                                         W\triangleq 
                                         \bigcap_{i=1}^m \ker(\la_i)
                               \end{equation}
                                      is a subspace complementary  to $V = \text{span}\ \xi$.
                                      
                                      Choose $v_{m+1}, \cdots, v_n$
                                      to be any basis of $W$.
                                      Then $\phi(v_1\wedge \cdots \wedge \hat v_i \wedge \cdots \wedge v_m \wedge v_j)=\pm\la_i(v_j)=0$
                                      for $1\leq i\leq m$ and $m+1\leq j\leq n$.
                                      Hence the proof gets completed.
                                                         \end{proof}

                            A powerful lemma to nicely control comass of a form by altering metics is the following.
                    \begin{lem}[Lemma 2.14 in \cite{HL1}]\label{hl2}
                    Let $\phi,\  V=span \ \xi$, and $W$ be given as in Lemma \ref{hl1}.
                    Consider an inner product $\big<\cdot, \cdot\big>$ on $\mathbb{R}^n$ such that $V\perp W$ and $\|\xi\|=1$. 
                    Choose any constant $C^2 > \bigl( \begin{smallmatrix} n\\ m\end{smallmatrix} \bigr) \|\phi\|^*$
                    and define a new inner product on $\mathbb{R}^n=V\oplus W$
                    by setting $\big<\cdot,\cdot\big>'=\big<\cdot,\cdot\big>_V+C^2\big<\cdot,\cdot\big>_W$.
                    Then under this new metric we have
                             $$\|\phi\|^*=1\ and\ \phi(\xi)=\|\xi\|=1.$$
             \end{lem}
  \begin{rem}\label{imp}
             If 
 $\phi(\xi)=\vartheta$ (positive) not necessarily one,
 one can apply {\it Lemma \ref{hl1}} to $\vartheta^{-1}\phi$ for
 $\|\phi\|^*=\vartheta,\ \|\xi\|=1\ and\ \phi(\xi)=\vartheta$
 by choosing $C^2>\vartheta^{-1} \big( \begin{smallmatrix} n\\ m\end{smallmatrix} \big)
\|\phi\|^*$.
              \end{rem}

                 In fact the ideas of Lemmas \ref{hl1} and \ref{hl2} also work for $m=1$.
                 To illustrate how they apply, we consider the situation in the figure.
                                                              \begin{figure}[h]
\begin{center}
\includegraphics[scale=0.8]{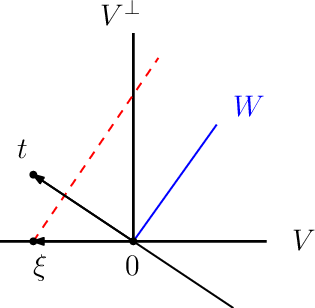}
\end{center}
\end{figure}
                      Although we do not need a metric a priori,  
                for simplicity, 
                      let  $\xi$ be of unit length and $t$ a vector with $\big<t, t\big>_{g_E}>1$ and $\big<\xi, t\big>_{g_E}=1$
                      in $(\R^n, g_E)$.
                      Then we can construct a form $\phi(\cdot)=\big<t,\cdot \big>_{g_E}$ 
                      which has the property $\phi(\xi)=1$.
                      
                              Now directly define $W\triangleq \ker(\phi)$.
                              Then $\{v\in \R^n\, :\,  \phi(v)=\big<\xi, v\big>=1\}=\{\xi\}+W$.
                             Define $g'(\cdot, \cdot)=\big<\cdot,\cdot\big>_V+C^2\big<\cdot,\cdot\big>_W$ as in Lemma \ref{hl2}.
                             Then, by Lemma \ref{comassf}
                             we have 
                             \[
                               \begin{split}
                              \|\phi\|^*_{g'} &=1/\min\{\| v\|_{g'}: v\ \text{is\ a\ vector\ at } 0 \text{ with}\ \phi( v)=1\}
                              \\
                              &=1/\sqrt{\|\xi\|^2+C^2\big<v-\xi, v-\xi\big>_W}
                               \\
                                          &\leq 1
                                \end{split}
                             \]
                             with equality if and only if $v=\xi$.
                              Note that here any $C^2>0$ is sufficient for the conclusion $\|\phi\|^*_{g'}=1$ and $\phi(\xi)=\|\xi\|_{g'}=1.$

                 It should be particularly emphasized that the distribution $W$ in the above for generating metric $g'(\cdot, \cdot)=\big<\cdot,\cdot\big>_V+C^2\big<\cdot,\cdot\big>_W$ to make $\|\phi\|^*_{g'}=1$ and $\phi(\xi)=\|\xi\|_{g'}=1$
                 is rigid and not allowed to vary at all.
                 For example, if one prefers to use $W=V^\perp$ and enlarge $C^2$ for $g'=g_E\big|_V+C^2g_E\big|_{V^\perp}$,
                 then no matter how large $C^2$ one chooses it is simply impossible to have $\|\phi\|^*_{g'}=1$. 
                 Similar phenomenon also occurs for simple $m$-form with constant coefficients when $m>1$.
                 
                 {\ }
                 
                 Although Lemma \ref{hl1} is independent on metric, 
                 sometimes when a prior metric gets involved we can have a coupled decomposition.
                      An immediate application to the case where $\phi$ is a calibration with constant coefficients in $(\R^n, g_E)$
                      and $\phi(\xi)= 1$ for some unit simple $m$-vector $\xi$  
                              can be derived.

                      \begin{lem}[Decomposition for calibration form] \label{L2}
                       Let $v_1, \cdots, v_n$ be an orthonormal basis of $(\mathbb{R}^n, g_E)$ with $\xi=v_1\wedge... \wedge v_m$
                       and $V = span\ \xi$.
                       Assume that $\phi$ is a calibration with constant coefficients which calibrates the oriented $V$.
                       Then $W$ in Lemma \ref{hl1} has to be $\text{span}\{v_{m+1}\, \cdots,  v_n\}$ 
                       and as in Lemma \ref{hl1} we have
                         \begin{equation}\label{decomp}
                                 \phi=v_1^*\wedge \cdots \wedge v_m^*+ \sum a_Iv_I^*,
                                  \end{equation}
                                  where $a_I=0$ whenever $i_{m-1}\leq m$. 
                                   Here $I=\{i_1,\cdots, i_m\}$ with $i_1<\cdots<i_m$.
                                                         \end{lem}

                      \begin{proof}
                      For $m\geq 2$, 
                      as in the proof of Lemma \ref{hl1},
                      linear maps
                                  $\la_i$ for $i=1,\cdots, m$ 
                                      can be defined.
                                      Now we must have
                          \begin{equation}\label{pair}
                                    \la_i(v_j)=(-1)^{m+i}
                                              \phi
                                              (  v_1\wedge\cdots \wedge \hat v_i\wedge \cdots \wedge v_m \wedge v_j)=\delta_{ij}
                         \end{equation}
                                      for $1\leq i\leq m$ and $1\leq j\leq n$ (not merely for $1\leq j\leq m$).
                                      Otherwise, 
                                      if $ \la_i(v_j)=a\neq 0$ in \eqref{pair}  for some $1\leq i\leq m< j\leq n$,
                                      then 
                                      consider the unit simple $m$-vector 
                                      $$\tilde \xi= v_1\wedge\cdots \wedge \ \left(\frac{1}{\sqrt{1+a^2}} v_i+\frac{a}{\sqrt{1+a^2}} v_j \right)\wedge \cdots \wedge v_m$$
                                      and 
                                      its pairing
                                      $\phi(\tilde \xi)=\sqrt{1+a^2}>1$
                                      contradicting with the assumption that $\phi$ is a calibration form.
                                      
                                     Due to the dimension reason,
                                 $W=\bigcap_{i=1}^m \ker(\la_i)=\text{span}\{v_{m+1}\, \cdots,  v_n\}$.
                                 
                                 When $m=1$, instead of constructing $\la_i$ we determine $W$ using $\phi$ directly and the same kind of arguments apply for the conclusion.
                      \end{proof}

           A key lemma that we establish in \cite{Z12} is the following.
           
              \begin{lem}[Comass control for gluing procedure]\label{CCGP}
                         For any $m$-form $\phi$, positive functions $a$ and $b$, and metrics $g_1$ and $ g_2$,
                         it follows
                                 \begin{equation}\label{gluemetrics}
                               \|\phi\|^*_{ag_1+bg_2}\leq\frac{1}{\sqrt{a^m\cdot \frac{1}{\|\phi\|^{*2}_{g_1}}+b^m\cdot \frac{1}{\|\phi\|^{*2}_{g_2}}}}
                                 \end{equation}
                         where $\frac{1}{0}$ and $\frac{1}{+\infty}$ are identified with $+\infty$ and $0$ respectively.
                         \end{lem}
   
   However, an apparent  dissatisfactory is that even with $\|\phi\|^{*2}_{g_1}= 1$ and $\|\phi\|^{*2}_{g_2}= 1$ 
   (except for $m=1$) one cannot draw the conclusion that                           $  \|\phi\|^{*2}_{(1-s)g_1+sg_2}\leq 1$ for $s\in(0,1)$.
   We shall obtain a better lemma regarding this in the next section.

                                  {\ }
                                  
                                  \section{New gluing result for the extension of calibration pairs}\label{S3}
                                      
                                        In this section we explain how to  extend calibration pairs.
                                               Case with some more general situation was studied,
                                               but  we shall  restrict ourselves to the simplest situation.
                                                  Let $(S,o)$  be  an oriented connected compact  submanifold with no boundary and with only one singular point $o$.
                                              Then  the  extension result of calibration pair  means  the following.
                                                 \begin{thm}[Theorem 4.6 in \cite{Z12}]\label{conecal}
          Given $(S,o)$ in $(X,g)$
          with $[S]\neq[0]\in H_m(X; \mathbb{R})$.
           If  $B_{\epsilon}(o;g)\cap S$
          is calibrated by a calibration in 
            some $\epsilon$-ball $(B_\epsilon(o;g),g)$ centered at $o$,
           then
            there exists a metric $\hat g$ coinciding with $g$ on $B_{{\frac{\epsilon}{3}}}(o;g)$ such that
            $S$ can be calibrated  in $(X,\hat g)$ by some calibration.
            \end{thm}

                                               Actually, we can simplify the gluing procedures in \cite{Z12}
                                               based on a satisfactory new result  in \S 3.1
                                               which can greatly improve the key control Lemma \ref{CCGP}.

                                           \subsection{A new gluing result.}
                                          In this subsection, we assume that 
                                          $\phi(\neq 0)$ is an $m$-form with constant coefficients in $\R^n$
                                        and
                                        that
                                          $\|\phi\|^{*2}_{g_1}\leq 1$ and $\|\phi\|^{*2}_{g_2}\leq 1$ for two metrics $g_1, g_2$ (with constant coefficients) on the vector space $\R^n$.

                                   A naive question is whether or not 
                                    the comass of $\phi$ remains no larger than one for the family of metrics $g(s)=(1-s)g_1+sg_2$.
                                    We confirm this affirmatively.
           
                                    \begin{thm}[Nice control for gluing metrics]\label{nice}
                                    Assume that $\phi$, $g_1,\, g_2$ and $g(s)$ are given as above.
                                    Then $\|\phi\|^*_{g(s)}\leq 1$.
                                    \end{thm}

                                     \begin{proof}
                                     There are three steps and the first two are the same as in the proof of Lemma \ref{CCGP} in \cite{Z12}.
                                     Let us first  observe that,
                                     according to Lemma \ref{comassf},
                                     we only need to focus on the smallest size of elements in
                                     \begin{equation}\label{Qset}
                                     \{\text{simple\ $m$-vector\ } Q \text{ at } 0 \text{ with}\ \phi(Q)=1\}
                                     \end{equation}
                                     under metric $g(s)$.
                                     
                                     The second is that, for each $P=span\, Q$,
                                     there exists an orthonormal basis $\{e_1, \cdots, e_m\}$ with respect to $g_1\big|_P$
                                     such that
                                     $g_2\big|_P$ is diagonalized as $diag\{\la_1,\cdots, \la_m\}$ under the basis and
    \begin{equation*}
          Q={\tt t}\,         e_1
                        \wedge\cdots\wedge 
                       e_m
     \end{equation*}
     where $\tt t$ is a nonzero real number ($|t|\geq 1$ since $\|\phi\|^*_{g_1}\leq 1$).
                     Hence, 
                                    \begin{equation} \label{Qn2}
                                     T(Q;s)\triangleq
                                             \left(
                                             \|Q\|_{g(s)}
                                             \right)^2
                                             =
                                            {{\tt t}^2} (1-s+s\la_1)\cdots (1-s+s\la_m).
                                               \end{equation}
                                                         By the pointwise assumptions  $\|\phi\|^*_{g_1}\leq 1$ and  $\|\phi\|^*_{g_2}\leq 1$,
                                                         it  follows that
                                                         both   
                                                         $ T(Q;0)$
                                                         and $ T(Q;1)$
                                                         are no less than one.

                                                                      Lastly, let us
                                                                      consider
                                                                      the function $F(X_1,\cdots, X_m)=(\Pi_{i=1}^m X_i)^{-1}$.
                                                                      A pleasant property of $F$
                                                                      is that it is strictly convex over the interior of the  first quadrant of $\R^m$.
                                                                          This can be checked very easily as
                                            \[                              \begin{split}
                                                                          \text{Hess}\, F
                                                                          &=F(X_1, \cdots, X_m)
                                                                          \cdot
                                                                         \left[
                                                                         \begin{array}{cccc}
\frac{2}{X_1^2}              & \frac{1}{X_1X_2}& \cdots &  \frac{1}{X_1X_m}\\
 \frac{1}{X_1X_2} & \frac{2}{X_2^2}   & \cdots &  \frac{1}{X_2X_m}\\
& \cdots & &\\
\frac{1}{X_1X_m}  & \frac{1}{X_2X_m} & \cdots &  \frac{2}{X_m^2}
\end{array}
                           \right]
                           \\
                          & =F(X_1,\cdots, X_m)\cdot \left\{diag\left(\frac{1}{X_1^2},\cdots, \frac{1}{X_m^2}\right)
                           +
                          {
                           \left(\frac{1}{X_1},\cdots, \frac{1}{X_m}\right)^T\left(\frac{1}{X_1},\cdots, \frac{1}{X_m}\right)
                           }
                           \right\}
                           \end{split}
                                           \]      
                                           is strictly positive.                       

                                                                          By choosing $X_i(s)=(1-s)1+s\la_i$,
                                                                         we have
                                                                          $F(X_1(s),\cdots, X_m(s)) =\frac{{\tt t}^2}{T(Q;s)}$
                                                                          and 
                                                                 \begin{equation}\label{good}
                                                                          \begin{split}
                                                                          F(X_1(s),\cdots, X_m(s)) 
                                                                                 &\leq 
                                                                                        (1-s) F(1,\cdots, 1)+ sF(\la_1,\cdots, \la_m)
                                                                                        \\
                                                                                        &=
                                                                                        \frac{ (1-s){\tt t}^2}{T(Q;0)}
                                                                                        +\frac{s{\tt t}^2}{T(Q;1)}
                                                                                        \\
                                                                                        & \leq
                                                                                        (1-s){{\tt t}^2}+ s{{\tt t}^2}
                                                                                        ={{\tt t}^2}.
                                                                          \end{split}
                                                               \end{equation}
                                                                       Therefore, we get $ T(Q;s)\geq 1$ for all $s\in[0,1]$
                                                                       and thus $\|\phi\|^*_{g(s)}\leq 1$.
                                     \end{proof}
           \begin{rem}\label{when=}
           If  for some $Q$ we have $\|Q\|_{g_1}= \|Q\|_{g_2}=1$,
           then the strict convexity
            leads to strict inequality $ \|Q\|_{g(s)}<1$ for $s\in(0,1)$
           unless 
           $\la_i=1$ for all $1\leq i\leq m$
           which implies
           $ \|Q\|_{g(s)}\equiv 1$ for $s\in [0,1]$.
           In particular, this tells that $P=span \, Q$ is calibrated by $\phi$ with respect to $g(s)$ for some $ s\in (0,1)$
           if and only if 
           $P$ is calibrated by $\phi$ with respect to both $g_1$ and $g_2$.
           The latter in fact forces $P$ to be calibrated by $\phi$ with respect to $g(s)$ for all $ s\in (0,1)$.
           \end{rem}

         For nonzero $m$-form $\phi$ and $g(s)=(1-s)g_1+sg_2$,
          Lemma \ref{CCGP}  means that
            $$
            \frac{1}{\|\phi\|^{*2}_{g(s)}}
                 \geq 
                 \frac{ (1-s)^m}{\|\phi\|^{*2}_{g_1}}
                 +
                 \frac{s^m}{\|\phi\|^{*2}_{g_2}}.
            $$
           In fact, according to the above proof, an improved control is the following.
           \begin{cor}\label{sqineq}
           As given in the above, we have
       $ {\|\phi\|^{*2}_{g(s)}}
           \leq (1-s)\|\phi\|^{*2}_{g_1} 
                      +
                      s\|\phi\|^{*2}_{g_2}
                      $.
           \end{cor}
           \begin{proof}
           As $\phi$ is not zero,
           there exists  $Q$ in \eqref{Qset} which realizes $\|\phi\|^{*2}_{g(s)}=\frac{1}{\|Q\|^2}$.
           Then, by dividing $\tt t^2$ on both sides of \eqref{good}
           and Lemma \ref{comassf},
           we get
           \begin{equation}\label{2ineq}
            {\|\phi\|^{*2}_{g(s)}}
            \leq 
             \frac{1-s}{T(Q;0)}
             +
              \frac{ s}{T(Q;1)}
          \leq (1-s)\|\phi\|^{*2}_{g_1} 
                      +
                      s\|\phi\|^{*2}_{g_2}.
           \end{equation}
           Moreover, one can also gain the strict inequality for all $s\in(0,1)$ if some $\la_i\neq1$.
           \end{proof}

           It is worth mentioning that Corollary \ref{sqineq} indicates that $\|\phi\|^{*2}_{(\cdot)}$ is a convex functional over the space of metric space.
           If one uses $F^a$ where $a>0$, then we can extend the result accordingly.
           
           \begin{cor}\label{expaineq}
           As assumed in the above, we have
       $ {\|\phi\|^{*2a}_{g(s)}}
           \leq (1-s)\|\phi\|^{*2a}_{g_1} 
                      +
                      s\|\phi\|^{*2a}_{g_2}
                      $.
           \end{cor}
      \begin{proof}
      Notice that, for every $a>0$,
             $F^a$ is strictly convex due to the same kind argument.
                                           So, for $Q$ in \eqref{Qset} realizing $\|\phi\|^{*2}_{g(s)}=\frac{1}{\|Q\|^2}$, besides \eqref{2ineq} we also have
                                        $$
            {\|\phi\|^{*2a}_{g(s)}}
            =\frac{F^a(X_1(s), \cdots, X_m(s))}{{\tt t}^{2a}}
            \leq 
             \frac{1-s}{T(Q;0)^a}
             +
              \frac{ s}{T(Q;1)^a}
          \leq (1-s)\|\phi\|^{*2a}_{g_1} 
                      +
                      s\|\phi\|^{*2a}_{g_2}.
          $$
              \end{proof}
           
           Therefore,  $\|\phi\|^{*2a}_{(\cdot)}$ is a convex functional over the space of metrics for all $a>0$.
           
           {\ }

                             \subsection{Simplified proof of Theorem \ref{conecal}}\label{spf}
                             Suppose that
                                        $\phi(\neq 0)$ is a calibration $m$-form with constant coefficients in $(\R^n, g_E)$.
                                          Namely, $\|\phi\|^*_{g_E}\leq 1$.
                                          
                                          Let $\xi\in \Lambda^m \mathbb{R}^n$ be a unit simple $m$-vector with
             $V = span\ \xi$.
                                     Suppose that $1\geq \phi(\xi)=\vartheta>0$.
                                     Then there exists a unique subspace $W$ in Lemma \ref{hl1}.
                                     By Lemma \ref{hl2} and Remark \ref{imp} 
                                           we have
                                           $\|\phi\|^*_{g'}=\phi(\xi)=\vartheta$ with respect to some $g'$.
                                           Hence, by Lemma \ref{nice}
                                           a better gluing for $g_E$ and $g'$
                                           is $(1-s)g_E+sg'$
                                          as
                                           $\|\phi\|^*_{(1-s)g_E+sg'}\leq 1$
                                           for $s\in[0,1]$.

                For an application to show how efficient  the result is, we  give a simplified proof of Theorem \ref{conecal} according to the following figure.
                              \begin{figure}[h]
\begin{center}
\includegraphics[scale=0.55]{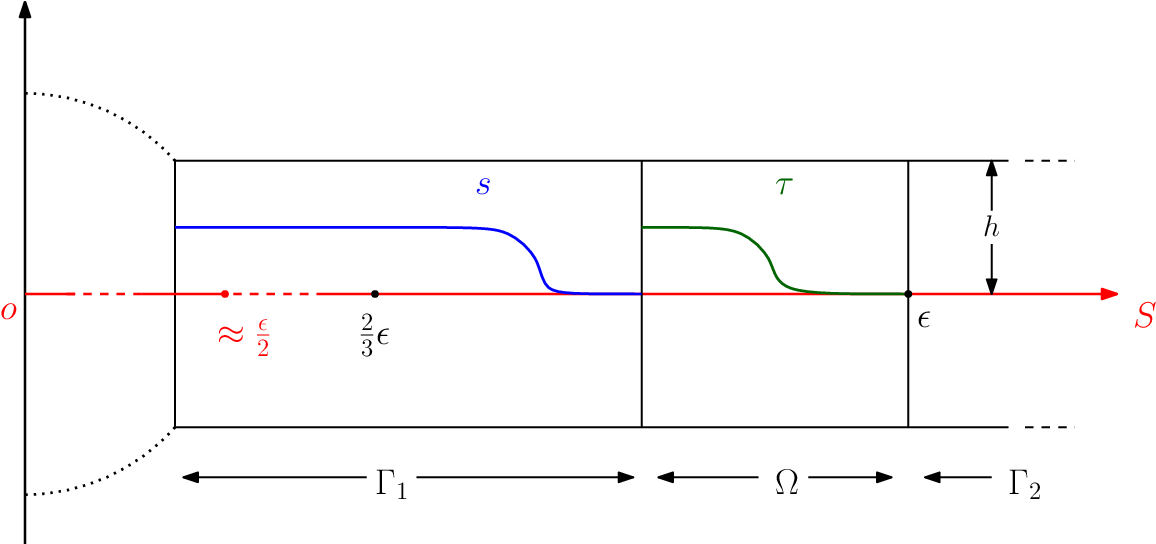}
\end{center}
\end{figure}
                              Let $\bold r$ be the induced distance from $o$ in $S$ by the a priori metric $g$.
                              Then regions $\Gamma_1,\, \Omega$ and $\Gamma_2$ with height $h$ here can be defined by certain ranges of $\bold r$.
                              
                              Note that the nearest projection $\pi$ to $S\cap \left(\Gamma_1\cup \Omega\cup \Gamma_2\right)$ can be well-defined on $\Gamma \triangleq \Gamma_1\cup \Omega\cup \Gamma_2$ when $h$ is small.
                              Denote the local calibration form in assumption by $\varphi$.
                              Since the region $\Omega$ can deform to a set of dimension strictly less than $m$,
                              it follows that $\varphi=d\psi$ on $\Omega$ for some smooth form $\psi$ of degree $m-1$.
                              The gluing of forms in \cite{Z12}
                     \begin{equation}\label{phiglue}
 \Phi\triangleq d(\tau\psi+(1-\tau)\pi^*(\psi|_S)),
\end{equation}
 where $\tau=\tau\circ \pi $ decreases from one to zero,
         provides a smooth form         on 
         $\Xi\triangleq B_{\frac{\epsilon}{2}}(o;g)\cup \Gamma.$
                              A remark here is that when $\epsilon$ is small enough 
                              $B_{\frac{\epsilon}{2}}(o;g)\cap S$ is approximately the set of points on $S$ which have $\bold r=\frac{\epsilon}{2}$, e.g. see \cite{LS} about the uniqueness of tangent cone at $o$. 
                             
                             Denote the oriented unit horizontal $m$-vector by $\overrightarrow Z$ in $\Gamma$
                             with the property that $\overrightarrow Z|_S$ agrees with the orientation of $S$
                             and span{$\overrightarrow Z$} is perpendicular to the fiber directions of $\pi$ with respect to $g$.
                             Clearly, $\overrightarrow Z$ is smooth and so is $\Phi(\overrightarrow Z)$.
                             Shrink $h$ if necessary to make $\Phi(\overrightarrow Z)>0$ in $\Gamma$.
                             Let 
                             $Z=\big(\Phi(\overrightarrow Z)\big)^{-1}        \overrightarrow Z.$
                             Then $\Phi(Z) \equiv 1$ in $\Gamma$.
                             Set 
                             $\tilde g=\big(\Phi(\overrightarrow Z)\big)^{\frac{2}{m}}g    $  in $\Gamma$
                             and
                             then $\|Z\|_{\tilde g}=1$.
                             By Lemmas \ref{hl1} and \ref{hl2} for $V=span\, Z$ and the compactness of $\bar \Gamma$,
                             one can get a smooth metric
                             $g'=\tilde g\big|_V+C^2\tilde g\big|_{W}$ with large $C^2$
                             such that 
                             $\|\Phi\|^*_{g'}\equiv 1$ in $\Gamma$.
                             The smoothness of $g'$ is due to the smoothness of $\Phi$ and $Z$
                             which induces the smoothness of $W$ and that of $\tilde g\big|_{W}$.
                             
                             Now we do the metric gluing for $g$ and $g'$ in $\Gamma_1$ using smooth function ${\bold s}={\bold s}({\bold r\circ \pi})$ which decreases from one to zero.
                             Since $\Phi=\varphi$ in $\Gamma_1$, we have $\|\Phi\|^*_{g}\leq 1$ in $\Gamma_1$.
                             By applying Theorem \ref{nice} pointwise,
                             we obtain $\|\Phi\|^*_{{\bold s}g+(1-{\bold s})g'}\leq 1$ in $\Gamma_1$.
                             Note that ${{\bold s}g+(1-{\bold s})g'}$ trivially extends to  a smooth metric $\check g$ in $\Xi$.
                             
                             As a result, we get a smooth calibration pair  $\Phi$ and $\check g$ in $\Xi$. 
                              Since $\Phi(\overrightarrow Z)=1$ in $S\cap \Gamma$,
                             it follows that  $g'$ and $g$ induce the same metric of the submanifold $S\cap \Gamma$.
                               Hence, by Remark \eqref{when=} for $Q=Z=\overrightarrow Z$ at every point in $S\cap \Gamma_1$,
                               we can conclude that $S$ is calibrated by $\Phi$ with respect to $\check g$.

                               A further extension from the calibration pair $(\Phi, \check g)$ on $\Xi$ to a global calibration pair $(\hat \Phi, \hat g)$ on $X$ 
                              and meanwhile making $S$ calibrated by $\hat \Phi$ can be made based  on the non-vanishing homology assumption.
                              One can follow the  step in \cite{Z12} and we omit it here.
                              
                                \begin{rem}\label{rGamma}
                   In fact by Lemma \ref{L2} the distribution of $W$ pointwise in $S\cap \Gamma_1$ and in $\Gamma_2$ is exactly fiber directions of $\pi$. 
                            \end{rem}

                            \begin{rem}\label{locF}
                           The method can also be applied to coflat calibration form with singular set being $\{o\}$.
                            \end{rem}

                              {\ }
                              
                           \section{Every Lawlor cone supports calibrations singular only at the origin} \label{S4}
                           In this section, we  show that every Lawlor cone can support coflat calibrations which are singular only at the origin of the ambient Euclidean space $(\R^{N+1}, g_E)$.
                           Here by Lawlor cone we mean a minimal cone which can be shown area-minimizing according to Lawlor's curvature criterion (in practice by the c-control or F-control to be explained below).
                           Such a cone naturally supports non-continuous dual calibrations.
                           Based on some analysis in  \cite{z}, we are able to construct some coflat calibrations resolving the discontinuity and  making the singular set exactly be $\{o\}$.
                           Besides,  a shortcut for the realization result in \cite{z} will be explained.
                           
                           First of all let us review some relevant materials from \cite{Law}.
                           \subsection{Introduction to Lawlor's curvature criterion. }\label{41}
                              The field of area-minimizing cones has been successively studied for decades due to their special central  role 
                              in several famous questions in differential geometry and geometric measure theory, e.g. see \cite{fle, de2, A, JS, BdGG, BL,  HL, HS, FK, Ch, Law0, Law, NS}.

                           The theory of calibrations mentioned in \S 2 is an effective way to study these objects.
                           Sometimes certain objects naturally associate with calibrations and further relate to dual area-nonincreasing projections.
                           As in \cite{HS},
                           Hardt and Simons
                           showed that for any given area-minimizing regular hypercone $C$ in $\R^{N+1}$
                           there exists a dilation-invariant foliation of the entire space $\R^{N+1}$ by minimal hypersurfaces with $C$   being the only singular leave.
                           Hence, the oriented unit volume form of the foliation induces a coflat calibration singular at the origin $o$ and also possibly in $C\sim o$.
                           Furthermore, the normal vector field integrates a dilation-invariant foliation by curves along which  an area-nonincreasing projection to the cone $C$ is then obtained.
                          
                           In  \cite{Law} Lawlor focused on a similar structure in some suitable angular neighborhood of a minimal cone $C(L^k)$ instead of in the whole space $\R^{N+1}$. 
                           If the boundary of the neighborhood can be mapped to the origin by the projection, 
                           then one can map everything outside the neighborhood to the origin. 
                           Thus an area-nonincreasing projection is gained and the area-minimality of the cone is proved.
                           
                           To be more precise, \cite{Law} introduces a  foliation by dilation-invariant curves 
                           and for further simplicity uses a uniform symmetric curve $\gamma$ through $x\in L$,
                           namely, for every $x\in L$
                           and any unit normal vector $v_x$ of $L\subset \mathbb S^N$ at $x$,
                           the curve  $\gamma(x, v_x)$ that Lawlor considered
                           has uniform expression in the polar coordinate
                           $r=r(\th)$.
                           The sufficient condition for the projection along foliation by $\{\R_+\cdot \gamma\}$
                           to be area-nonincreasing is the inequality
                           \begin{equation}\label{ineq}
             \frac{dr}{d\theta}\leq r\sqrt{r^{2k+2}\cos^{2k}\theta 
                                                 \inf_{x\in L, \, v_x\in N_xL}
                                                             \left(
                                                             \det
                                                             \left(
                                                                        \textbf{I}-(\tan\theta) \textbf{h}_{ij}^{v_x}
                                                                        \right)
                                                             \right)^2
                                                             -1
                                                             }
           \end{equation}
             with 
             $r(0)=\|x\|=1$.
 Here 
             $\textbf{h}_{ij}^v$ is the second fundamental form at $x$ for $v$.
             To map the boundary of the region where the projection is defined,
             one wishes to have $\lim_{\theta\uparrow \theta_0}r(\theta)=\infty$ for some vanishing angle $\th_0$.
             To get the narrowest, \eqref{ineq} with equality was studied.
            Set $t=\tan \th$, 
             $h(t)=(r(\th)\cos \th)^{-k-1}$
             and 
             $$p(t)=  \inf_{x\in L, \, v_x\in N_xL}
                                                             \det
                                                             \left(
                                                                        \textbf{I}-t \textbf{h}_{ij}^{v_x}
                                                                        \right)
                                                                        =1+p_2 t^2+\cdots
                                                                     \text{   with negative }p_2.
$$
            Then the equivalent version of \eqref{ineq} in creating calibration form 
            was obtained in \cite{Law}
            \begin{equation}\label{tineq}
                          \left( h(t)-\frac{t}{k+1}h'(t)  \right)^2+\left(\frac{h'(t)}{k+1}\right)^2\leq (p(t))^2,\ \ \ \ \ \ h(0)=1 .
                          \end{equation}
          Note that the candidate calibration  has the nice form  $d(h\psi)$ for some peculiar smooth form $\psi$ (see \cite{Law} for more details)
          and the \eqref{tineq} is just the requirement for comass no larger than one.
  As for the boundary condition, $\lim_{\theta\uparrow \theta_0}r(\theta)=\infty$ exactly corresponds to the situation that $h$ hits the $t$-axis at $t_0=\tan \th_0$ for the first time.
                         Moreover,    once the normal radius $R(L)$ of $L$ in $\mathbb S^N$ is larger than twice the vanishing angle, 
                             one can get a well-defined area-nonincreasing projection to $C(L)$. 
                             
                                 Lawlor's curvature criterion is the following.
                                                \begin{thm}[\cite{Law}]\label{LC}
                                                Given a regular minimal cone $C(L)\subset \R^{N+1}$,
                                                if the vanishing angle $\theta_0$ exists
                                                and $\theta_0\leq \frac{R(L)}{2}$,
                                                then $C(L)$ is area-minimizing.
                                                \end{thm}

                In practice, Lawlor used the control
                                            \begin{equation}\label{control}                                                                                                                                                                              (1-\alpha t)e^{\alpha t}
                                                                       <
                                                                       F(\alpha, t, k+1)
                                                                       \leq
                                                                        p(t)
                                              \end{equation}
                                                      where 
                                                                                                                                                    $
                                                      \alpha=
                                                      \sup_{x\in L, v_x\in N_xL}
                                                      \|     \textbf{h}_{ij}^{v_x}      \|
                                                      $
                                                       and
                                     $
                                                      F(\alpha, t, k+1)=
                                                                 \left(
                                                                 1-\alpha t\sqrt{\frac{k}{k+1}}
                                                                  \right)
                                                                  \left(
                                                                  1+\frac{\alpha t}{\sqrt{k(k+1)}}
                                                                  \right)^{k}
                                   $
                                        to derive two kinds of vanishing angles by employing \eqref{control} to replace $p(t)$
                                                      $$
                                                      \theta_c( k+1, \alpha)
                                                      >
                                                      \theta_F( k+1, \alpha)
                                                      \geq
                                                       \theta_0( C(L))
                                                       .
                                                      $$
                                                      So, if either 
                                                      $\theta_F$
                                                      or 
                                                      $\theta_c$ is less than $\frac{R(L)}{2}$,
                                                      then $C(L)$ is area-minimizing.

                            \subsection{Proof of Theorem \ref{T1}. }\label{42}
                            By a Lawlor cone we mean the smallest vanishing angle $\th_0(C(L))$ is strictly less than $\frac{R(L)}{2}$ in the practical sense.
                            All area-minimizing cones gained  by applying Lawlor's curvature criterion with F-control or c-control satisfy $\th_0(C(L))<\frac{R(L)}{2}$.
                            Even if there existed some cone $C(L)$ with $\th_0(C(L))=\frac{R(L)}{2}$ in the theoretical sense,
                            one cannot detect it through Lawlor's curvature criterion as the solution $h_0(t)$ which touches the $t$-axis at $t=\tan \th_0$
                            has  an unstable behavior around $t=0$ (in the stage of second order) and simply becomes unobservable using Lawlor's vanishing angle table or any other numerical methods.
                            
                            As a result, each Lawlor cone must have some solution $h$ to \eqref{tineq} with equality everywhere which touches the $t$-axis at some  $t=\tan \th_3>\tan \th_0$ as in the following graph,
                            and moreover, $\th_3< \frac{R(L)}{2}$ as well.
                              \begin{figure}[ht]
                                    \begin{center}
                                    \includegraphics[scale=0.6]{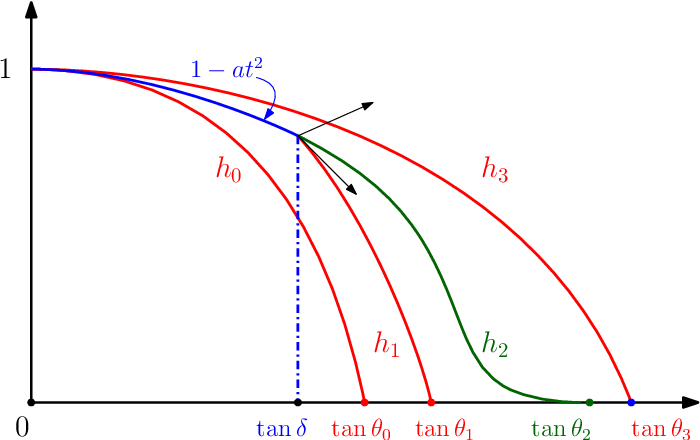}
                                    \end{center}
                            \end{figure}
             \begin{pfT1}                                                                  
                                                      For  point $(t, y)$ with $0<y<\sqrt{t^2+1}p(t)$, 
                                                      the inequality \eqref{tineq}
                                                       leads to
                                                                \begin{equation}\label{interval}
                                                                                          \dfrac{k}{t^2+1}\left(ty-\sqrt{(t^2+1)p^2(t)-y^2}\right)
                                                                                                         \leq h'(t) 
                                                                                                                    \leq \dfrac{k}{t^2+1}\left(ty+\sqrt{(t^2+1)p^2(t)-y^2}\right).
                                                                 \end{equation}
Since 
  $
  \dfrac{\partial}{\partial y}\left(ty-\sqrt{(t^2+1)p^2(t)-y^2}\right)=t+\dfrac{y}{\sqrt{(t^2+1)p^2(t)-y^2}}>0,
  $
%
 in order to reach zero fastest, the solution $h_0$ must satisfy
                     \begin{equation}\label{ceq}
                          h'(t)=\dfrac{k}{t^2+1}\left(th-\sqrt{(t^2+1)p^2(t)-h^2}\right),
                          \end{equation}
      with $h(0)=1.$
                             At the stage of second order,
                             $h_0(t)=1-a_{max} t^2+\cdots$
                             and
                             $h_3(t)=1-a_{min} t^2+\cdots$
                         where    
                             $a_{max,min}=\frac{k}{4}\left(k-2\pm\sqrt{(k-2)^2+8p_2}\right)$.
                             Hence, as in \cite{z}
                             we can
                             choose 
                                $a_{min}<a<a_{max}$
                               and
                               sufficiently small $\delta$ 
                               such that for, all $t\in(0,{\tan\delta}]$,                       
                      $1-at^2$ 
                      satisfies \eqref{tineq} which lies between $h_0$ and $h_3$.
                      Moreover it can further extend to $h_1$ by \eqref{ceq}.
                      
                      Note that
                      $h_0$ is strictly deceasing,
                      that is to say $h(t)<p(t)$ in the interval $(0,\tan \th_0)$.
              So, one can arrange $a$ and $\delta$ to make  $h_1$ strictly decreasing as well.
              Now one gets  $h_1(t)<p(t)$ and around the graph of $h_1$ in the figure
              the inequality \eqref{interval} tells that the slope  $h'(t)$ of a solution to the inequality of \eqref{tineq} is allowed to vary from a strictly negative number to a strictly positive number as illustrated in the figure.
             Then one can deform $h_1$ to $h_2$ which smoothly joins 
                                   $1-at^2$
                                   and the $t$-axis with its graph touching the axis at $\tan \th_2$  as close as one wishes to the point $\tan \th_1$,
                                   and moreover satisfies the calibration inequality \eqref{tineq}.
                                   The reason is 
                                   that in any small interval $[ \delta , \epsilon +\delta)$ one can extend $1-at^2$ smoothly and make the extension  fulfill 
                                   \eqref{interval} in $(\delta, \epsilon + \delta)$
                                   and
                                   \eqref{ceq} near $t=\epsilon +\delta$.
                                   This allows us to further extend it  by  \eqref{ceq}
                                   and it will hit the $t$-axis at a point $\tan \hat \th$ close to $\tan \th_1$.
                                   Similarly as what has been done around $t=\delta$,
                                   one can deform 
                                   this solution of inequality \eqref{tineq}
                                   around $\tan \hat \th$ to make it smoothly tangent to the $t$-axis at $\tan \th_2$ 
                                       such that  the resulting $h_2$ also satisfies \eqref{tineq} and $\tan \th_2$  can be as close to $\tan \hat \th$ as one wants.
                                   
                                   As a result, the corresponding calibration $\phi=d(h\psi)$ is smooth in $C(L)\sim o$ and also in the boundary of its $\th_2$-angular neighborhood $U(\th_2)$ with vanishing $d(h\psi)$ outside $U(\th_2)$.
                                   Thus, we finish the proof.
             \end{pfT1}

             \begin{rem}
             We would like to point out that one can also consider smooth extensions as follows.
              $$\begin{minipage}[c]{0.55\textwidth}
                              \includegraphics[scale=0.5]{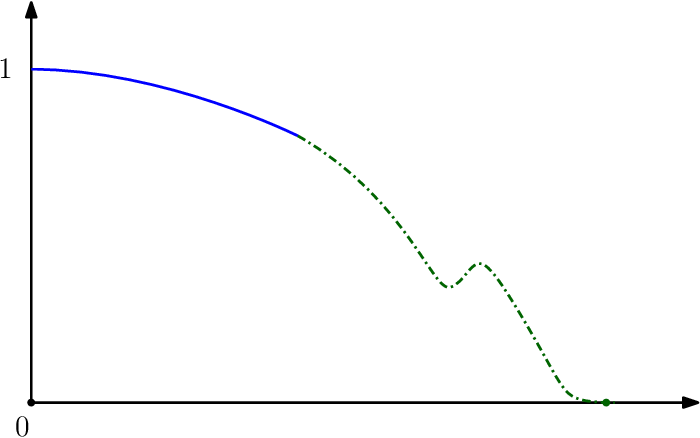}
                              \end{minipage}%
                          \begin{minipage}[c]{0.55\textwidth}
                           \includegraphics[scale=0.51]{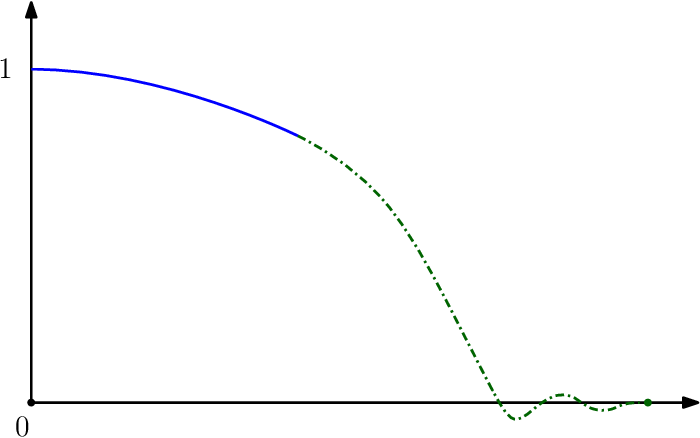}
                           \end{minipage}$$
              They  cause no problem for gaining smooth calibrations away from $o$ for the conclusion in Theorem \ref{T1} but 
              the latter cannot induce a  global area-nonincreasing projection to $C(L)$ faithfully corresponding to the figure while the former can.
             \end{rem}
             
{\ }

                           \subsection{A shortcut for the realization result in \cite{z}. }\label{43}
                           How to get  some closed homologically area-minimizing singular hypersurface  in some Riemannian manifold
                           remained a long-term standing question
                                      until 
                                      \cite{NS}
                                      where
                                      N. Smale found first family of such examples utilizing many tools from geometric measure theory and geometric analysis.
                      
                           Based on \cite{Z2} where the author showed every homogeneous area-minimizing hypercone supports coflat calibrations singular only at the origin, 
                           Theorem \ref{conecal} (Theorem 4.6 in \cite{Z12}) can create examples of closed homologically area-minimizing singular hypersurfaces  in some Riemannian manifolds.

                           In \cite{z} 
                           we study more general realization question:
                           \begin{quote}
{\it Can any area-minimizing cone be realized as a tangent cone at a point
of some homologically area-minimizing  {\em compact} singular submanifold?}
\end{quote}
And we answered affirmatively for every oriented Lawlor cone $C(L)$.
                           
                           The ideas of \cite{z} include a reduction argument of the question to local by Allard's monotonicity,
                           some modification on Lawlor's dual non-continuous calibration for a calibration which is smooth in $L(C)\sim 0$,
                           and a series of mollifications on Lipschitz potential of the non-continuous calibration by virtue of  Nash's embedding theorem. 
                           Suitable combinations among them also play a role for other related questions.
                        
                        Let us recall that,
                      as in \cite{NS},  we consider in \cite{z}  
                          $$\Sigma_C\triangleq (C\times \mathbb R)\cap \mathbb S^{N+1}$$ in $ \mathbb R^{N+1}\times \mathbb R$.
                         Let $M$ be an embedded oriented  connected compact $m$-dimensional submanifold in
                         some $(N+1)$-dimensional oriented compact manifold $T$ with $[M]\neq [0]\in H_m(T;\mathbb Z)$.
                         Within smooth balls around a point of $M$ and a regular point of $\Sigma_C$ respectively
                         one can connect $T$ and $\mathbb S^{N+1}$, $M$ and $\Sigma_C$ simultaneously through one connected sum.
                         Denote by $X$ and $S$ the resulting manifold and submanifold (singular at two points $p_1$ and $p_2$).
                         Then $[S]\neq [0]\in H_m(X;\mathbb Z)$.

                       Now we would like to point out here that Theorem \ref{T1} and Theorem \ref{conecal} 
                       (for $C$ being a Lawlor cone and $S$ having exactly two singular points $p_1, p_2$,
                       the argument in \S  \ref{spf} is still valid since $\Phi$ and $\check g$ in $\Gamma_2$ are uniquely determined by the initial $g$ cf. Remark \ref{rGamma}
                       so things match together well)
                        immediately open a shortcut to conclude that every oriented Lawlor cone $C(L)$ can be realized to the realization question.
                           
                           {\ }
                           
 \section{Minimal product and a proof of Theorem \ref{T2}}\label{S5}
                              
                              We will give a proof of Theorem \ref{T2} in this section.
                               Actually, the result holds for continuous calibration as well.
                               
                               To begin let us briefly recall the minimal product construction.

                               \subsection{Minimal product. }
                               Given minimal submanifolds $\{L_i\}$ in $\mathbb S^{N_i}$ respectively for $i=1,\cdots, n$,
                               a nice algorithm to generate new minimal submanifold  in sphere is the following
                         \begin{equation}\label{prod}
         L_1\times\cdots
         \times L_n
         \triangleq
          \left(\lambda_1 L_1, \cdots, \lambda_n  L_n\right)\subset \mathbb S^{N_1+\cdots +N_n+n-1}.
       \end{equation}
         where $\lambda_i=\sqrt{\frac{k_i}{k}}$ with $k_i$ the dimension of $L_i$ and $k=\sum_{i=1}^n k_i>0$.
              
              This structure may be already known by some experts for quite long period.
              Special situations with all inputs being full spheres or all being R-spaces have been studied, e.g. see \cite{Law, OS} and references therein.
              A simple proof of the structure for general minimal submanifolds through Takahashi Theorem can be found in the book \cite{X}
       \footnote{This reference came into the author's notice in  2020, thanks to the presentation copy (new edition)  from Professor Xin  to the author in 2019.}.
              In 2016 (arXiv chronology), the structure was rediscovered independently by \cite{TZ} and \cite{CH} using different methods.
              Readers may choose either familiar way to understand it
             and the complete information of second fundamental form can be found in \cite{TZ}.

                               \subsection{Proof of Theorem \ref{T2}}
                                                              A preparation for the proof is the following wonderful result by Jim Simons (also see an equivalent version in \cite{de2}).

                               Let $M$ be a minimal hypersurface of $\mathbb S^n$.
                               Choose an orientation and denote the unit normal vector  at $x\in M$ by $N(x)$.
                             The parallel replacement of $N(x)$ at the origin is called $x^*$.
                               Use $M^*$ to stand for the image of $M$ under the mapping $x\longmapsto x^*$.

                               \begin{thm}[Version in \cite{JS}]\label{JS}
                               Suppose that $M$ is a closed minimal hypersurface in $\mathbb S^n$.
                               Then either $M^*$ is a single point, in which case $M=\mathbb S^{n-1}$,
                               or $M^*$ lies in no open hemisphere of $\mathbb S^n$.
                               \end{thm}
                               
                               Now we begin our proof of Theorem \ref{T2}.
                               
                               \begin{pfT2}
                               Assume that  each $L_i\subset \mathbb S^{N_i}$ is a closed minimal submanifold and $L_1$ of codimension one in $\mathbb S^{N_1}$.
                               Suppose that $\phi$ is a continuous calibration defined in  the entire space $(\R^{N+1}, g_E)$
                               which calibrates  the cone $C(L_1\times\cdots\times L_n)\subset \R^{N+1}$ over minimal product  $L_1\times\cdots\times L_n$.
                               
                               The first observation is that the restriction $\phi_o$ of $\phi$ at the origin $o$ calibrates $C(L_1\times\cdots\times L_n)$.
                               This is due to  the continuity of calibration $\phi$ and that the tangent cone of a cone is itself.
                               And one can apply argument by contradiction to get the observation rigorously.
                               
                               Based on the observation, we know that $C(L_1\times\cdots\times L_n)$ must be calibrated by $\phi_o$
                               which is a calibration with constant coefficients.
                               Fix 
                               $x_i\in L_i$ 
                               and 
                               let $\xi_i$  be    the oriented unit tangent $k_i$-vector of $T_{x_{i}}L_{i}$    for $i=2, \cdots, n$.
                               Then we shall focus on 
                             the calibration $(k_1+1)$-form 
                             $\Phi_o\triangleq \left(\xi_2\wedge \cdots \wedge \xi_n\right) \lrcorner \phi_o$ 
                             with constant coefficients.
 Denote $\mathcal P^*(\Phi_o)$ to be the pull-back of $\Phi_o$ to the linear subspace $\R^{N_1+1}$ via the inclusion $\mathcal P: \R^{N_1+1}\hookrightarrow \R^{N+1}$ by $\mathcal P(x)=(x, 0, \cdots, 0)$.
                         Note that for $x\in L_1$ we have   the oriented unit tangent $k_1$-vector $\xi=\xi(x)$ of $T_{x}L_{1}$ and $\xi \wedge x$ is the volume form of $C(L_1)\subset \R^{N_1+1}$ (reverse the orientation if need be, here and in the sequel).

                         Without loss of generality, assume that $L_1$ is connected (or restrict ourselves to one connected component of $L_1$).
                         We aim to show that  evaluation of $\mathcal P^*(\Phi_o)$ with $\xi \wedge x$  is a common nonzero constant for all $x\in L_1$.
                         First note that evaluation of $\mathcal P^*(\Phi_o)$ with $\xi \wedge x$ has the same evaluation value of $\Phi_o$ with $\xi \wedge x$.
                          (Here the second $\xi \wedge x$ means the $(k_1+1)$-vector $\mathcal P_\#(\xi \wedge x)$ in $\R^{N+1}$.)
                         Moreover, 
                          $\Phi_o$ (or $-\Phi_o$) calibrates all the $(k_1+1)$-vectors $\{\xi \wedge p\}$ (short for $\{(\mathcal P_\#\xi) \wedge p\}$)
                          where $p=p(x)=(\la_1 x, \la_2 x_2,\cdots, \la_n x_n)$ with $x$ running in $L_1$.
                          So, for each $x\in L_1 $,
                           we can extend an orthonormal basis $ \{v_1,\cdots, v_{k_1} \}$ of span $\xi$ to an orthonormal basis 
                          \begin{equation}\label{basis}
                          \Big\{v_1,\cdots, v_{k_1},  \, v_{k_1+1},\, \ v_{k_1+2}, \,\cdots, v_N\Big\}
                           \end{equation}
                           of $\R^{N+1}$
                          with $v_{k_1+1}=p$ and 
                           $v_{k_1+2}=p^\perp=p^\perp(x)$  a unit orthogonal vector to $p$ in the plane span$\{p, x\}$.
                           Now by applying Lemma \ref{L2}
                           we have the decomposition 
                            \begin{equation}\label{deC}
                                 \Phi_o= \pm v_1^*\wedge \cdots \wedge v_{k_1+1}^*+ \sum a_Iv_I^*,
                                  \end{equation}
                                  where $a_I=0$ whenever $i_{k_1}\leq k_1+1$. 
                                  For nonzero $a_I$, the corresponding multi-index
                                  $I$ must have at least two indices  strictly larger than $k_1+1$.
                               Note that among elements of \eqref{basis} only $v_{k_1+1}$ and $v_{k_1+2}$ have nonzero evaluations with $x$
                               but the coefficient $a_I$ vanishes for $I=1\, 2\, \cdots\, k_1\, \widehat{\, (k_1+1)\, }\, (k_1+2)$.
                            %
                                  %
                                       As a result, 
                                       it follows by \eqref{deC} that
                       \begin{eqnarray}
                       \left(\mathcal P^*(\Phi_o)\right)(\xi\wedge x)
                       &
                       =
                       &
                                                                  \Phi_o(\xi\wedge x)
                                                                  \nonumber
                                                                  \\
                                            &
                                            =
                                            &
                                            \left(\pm v_1^*\wedge \cdots \wedge v_{k_1+1}^*\right)(\xi\wedge x)
                                            \nonumber
                                            \\
                                            &
                                            =
                                            &
                                            \pm\la_1
                                 \label{evaC}         
                          \end{eqnarray}
                          for all $\xi(x)\wedge x$ where $x\in L_1$.
                          
                            By the assumption, we know that  $k_1+1=N_1$.
                                   So  $\mathcal P^*(\Phi_o)$  is a simple form in $\R^{N_1+1}$.
                          In particular,
                          by \eqref{evaC} and the connectedness of $L_1$ (or its connected component), 
                          the fixed unit normal dual vector to $\mathcal P^*(\Phi_o)$ has  constant angle $\arccos \la_1$ or $\pi-\arccos \la_1$ with respect to the unit normal of $\xi\wedge x$ for all $x\in L_1$.
                          Namely, $L_1^*$  is contained in an open hemisphere of $\mathbb S^{N_1}$ but this cannot happen according to Theorem \ref{JS}.
                          
             Therefore, we finish the proof by contradiction argument,
             with the conclusion that
             whenever someone among $\{L_i\}$ is a minimal hypersurface in sphere,  
                                   the cone $C(L_1\times\cdots\times L_n)$ cannot be calibrated by any globally defined continuous calibration form.
                               \end{pfT2}
                               
      \begin{rem}
      There are many examples of area-minimizing cones of high codimension which can support calibration with constant coefficients, e.g. see the coassociative cone and special Lagrangian cones in \cite{HL}.
      The spiral minimal product algorithm in \cite{L-Z}  can  produce a plethora of new examples of special Lagrangian cones based on given ones.
      
     However, it is unclear to the author at this moment whether one can find a cone over some  minimal product (of inputs of positive dimensions)  which can support  continuous calibrations. 
     If there were no such kind of examples at all, then it indicates that the minimal product structure can automatically produce cones which support no continuous calibrations.
      \end{rem}

                              {\ }
                              
                               \section{Detect duality obstruction in higher codimension}\label{S6}
                                    Before discussion on the situation of higher codimension, 
                                    let us review main steps for creating Example 3 in \cite{Z12}
                                    which essentially displays the duality obstruction of calibration in the smooth category for codimension one.
                                    
                                    \subsection{Review on the duality obstruction in codimension one. }   \label{61} 
                                        Roughly speaking, there are three steps.
                                        The first is to construct closed singular  submanifold $(S,o)$   of nonzero  real homology class  
                                          in some manifold $X$ with tangent cone $C$ at $o$ being a homogeneous area-minimizing hypercone (say the Simons cone $C_{3,3}$).
                                        Then by \cite{Z2} and Theorem \ref{conecal}        
                                        we can get a coflat calibration  $\hat \Phi$ singular only at $o$ with respect to some smooth metric $\hat g$ on $X$ 
                                        such that $S\sim o$ is calibrated by $\hat \Phi$.
                                        
                                        The second step is to fix a closed smooth submanifold $M$ in $[S]$ which has empty intersection with $S$
                                        and conformally change metric $\hat g$  to $h$ in sufficiently small neighborhood of $M$
                                        such that $M$ and $S$ are the only two mass-minimizing representatives in the class $[S]$ with respect to $h$.
                                        This can be made rigorously by calibration arguments, for details see \cite{Z12}.
                                        
                                        The last step is to observe that, under $h$, although $M$ is homologically mass-minimizing,
                                        it cannot be calibrated by any smooth/continuous calibration form.
                                        Since otherwise the singular hypersurface $S$ must be calibrated 
                                        according to the fundamental theorem of calibrated geometry, Theorem \ref{hl}.
                                        As a result, the non-flat tangent cone of $S$ at $o$ then must be calibrated by a calibration with constant coefficients in $(T_o X, h_o)$
                                        and this is impossible.

                                        From the steps in the above, the duality obstruction of calibration in the smooth category is finally due to codimension one.
                                        
                               {\ }
                               
                               \subsection{Generalization to higher codimension} \label{62}
                               Let us mention that the first two steps in the above can be done without difficulty for higher codimension.
                               
                               The reason is that one can apply the construction in \S \ref{43}
                               with $T^{N+1}=S^1_1\times \cdots \times S^1_N$
                               and the submanifold  being the torus $T^m=S^1_1\times \cdots \times S^1_m\times \{q_{m+1}\}\times \cdots \times \{q_{N+1}\}$
                               where $q_j\in S^1_j$ for $m<j\leq N+1$.
                               Hence, for any $m$-dimensional cone $C\subset \R^{N+1}$ which can be calibrated by some coflat calibration singular only at the origin,
                               the connect sum of $T^{N+1}$ and $\mathbb S^{N+1}$ around regular points of $T^m$ and $\Sigma_C$
                               gives us $X$ and $S$ in the first step.
                                
                                As already explained in \S \ref{43}, Theorem \ref{conecal} can be applied to the resulting $S$ with two singular points.
                                Let $(\hat \Phi, \hat g)$ be the calibration pair on $X$ and $\hat \Phi$ calibrates $S$.
                                Further, one can choose $M$ to be $S^1_1\times \cdots \times S^1_m\times \{q'_{m+1}\}\times \cdots \times \{q'_{N+1}\}$ where some $q'_j\neq q_j$.
                              Then by exactly the same reason
                              one can conformally change $\hat g$ to $h$ such that $M$ and $S$ are the only two mass-minimizing in the class $[S]$.
                              
                              Again, by the fundamental theorem of calibrated geometry, if $M$ is calibrated by some smooth/continuous calibration form,
                              then the tangent cone $C$ of $S$ at  singular point $p_1$ (or $p_2$) has to be calibrated by a calibration with constant coefficients.

                            {\ }
                            
                            Now we are ready to state what kind of local model $C\subset \R^{N+1}$ can lead to duality obstruction in higher codimension in the last step.
                            Actually there are quite many.
                            The following is a recent general configuration result for mass-minimizing cones in Euclidean spaces.
                            
                             \begin{thm}[Theorem 1.1 of \cite{Z25}]\label{main}
                                Given embedded closed minimal submanifolds $\{L_i\}$ respectively in Euclidean spheres with $i=1,2,\cdots, n$,
                                every cone over the minimal product $L$ of sufficiently many copies among these $\{L_i\}$ is area-minimizing.
                                \end{thm}
                            
                            Another neat configuration result is the following. 
                               \begin{thm}[Theorem 1.3  of \cite{Z25}]\label{addcor}
                               Let $L$ be an embedded closed  minimal subamanifold in $\mathbb S^n$.
                              Then the cone over minimal product $L\times \mathbb S^{d}$ in $\R ^{n+d+2}$ is area-minimizing when $d$ is sufficiently large.
                               \end{thm}

                               Both results are proved by Lawlor's criterion.
                               So our Theorem \ref{T1} here asserts the existence of  coflat calibration forms singular only at the origin for each Lawlor cone in Theorem \ref{main} and Theorem \ref{addcor}. 
                               
                               If a Lawlor cone $C( L_1^{\times s_1} \times \cdots \times L_n^{\times s_n})$ from Theorem \ref{main}
                               has some input, say $L_1$, a minimal hypersurface in $\mathbb S^{N_1}$ with $s_1>0$,
                               then by our Theorem \ref{T2}
                               we know that it cannot be calibrated by any smooth/continuous calibration form.
                               
                               Let $L$ be an arbitrary  closed minimal submanifold in $\mathbb S^{N}$
                               then $C(L\times \mathbb S^{2025}\times \mathbb S^d)$
                               is a Lawlor cone when $d$ is sufficiently large.
                               Note that the minimal product $\mathbb S^{2025}\times \mathbb S^d$ is a minimal hypersurface in $\mathbb S^{2026+d}$.
                               So by  our Theorem \ref{T2} viewing $L\times \mathbb S^{2025}\times \mathbb S^d=L\times \left(\mathbb S^{2025}\times \mathbb S^d\right)$
                               we know that $C(L\times \mathbb S^{2025}\times \mathbb S^d)$ cannot be calibrated by any smooth/continuous calibration form.

                               Therefore, based on the such Lawlor cones as local models we can follow the above three steps to detect   the duality obstruction of calibrations in the smooth category for the situation of higher codimension. 
                              In summary, 
                               key points 
                               are 1. existence of coflat calibration singular only at the origin in local model by Theorem \ref{T1} in the second step
                               and 2.  the contradiction by Theorem \ref{T2} in the last step.
                               
                          {\ }
                             
                             \subsection{Another way to generate duality obstruction in high codimension beyond $\bf{ (\ast)}$.}
                             Recall that
                             in \cite{Z12}
                             we remark after Example 2 there that certain Cartesian cross-product
                             can give us realizations to the realization question in \S \ref{43} for models with more complicated singular set. 
                             
                             To be more precise, we quote Remark 4.12 in \cite{Z12}:
                             \begin{quote}
                             By cross-products examples with more complicated singularity can be generated.
For instance, 
$S\times S$ with singularity $S\vee S$ is calibrated by a coflat calibration
with singular set 
$
S\vee S
$ 
in the cartesian product $(X, g)\times (X, g)$.
                             \end{quote}
                             This in no way means that the Cartesian product of any two calibrations is again a calibration.
                             In fact a recent result of \cite{Zust25} shows the opposite and thus gives a negative answer to a long-standing Federer question (of comass version).
                             The reason for the conclusion in the quotation is that our singular $(S,o)$ in Example 2 is of codimension one
                                       and
                                       the coflat calibration has to be a simple form of co-degree one away from the singularity. 
                                       As a result, in this situation the product of calibrations does produce a coflat calibration with singular set being $X\vee_{(o,o)} X$.
                                       Since $(S,o)\times (S,o)$ with  singular set $S\vee_{(o,o)}  S$ possesses dimension larger than that of $X\vee_{(o,o)} X$,
                                       we see that $(S,o)\times (S,o)$ is calibrated by the resulting coflat calibration 
                                       and
                                       the coflat version of fundamental theorem of calibrated geometry now applies.

                             Therefore, take smooth $M$ in [S] with $M\cap S=\emptyset$ as in Example 3 of \cite{Z12}.
                             As briefly explained in \S \ref{61} we can conformally change $\hat g$ to $h$ such that
                             $M$ and $S$ are the only two minimizers in $[S]$ with respect to $h$.
                             Actually, what we performed in \cite{Z12} is to make both $M$ and $S$ be calibrated by some coflat calibration $\tilde \Phi$ with singular set being $\{o\}$. 
                             Hence, following the preceding paragraph here we can see that
                             $$
                             M\times M,\, \ \ \ M\times S,\, \ \  \ S\times M,\, \ \  \ S\times S
                             $$ 
                             are calibrated by the coflat calibration $\tilde \Phi\wedge \tilde \Phi$ in $(X, h)\times (X,h)$.
                             Thus $M\times M$ is homologically mass-minimizing
                             but cannot support any smooth calibration form $\varphi$ globally defined on $(X, h)\times (X,h)$.
                             Otherwise,  by the fundamental theorem of calibrated geometry $\varphi$ also calibrates $M\times S$ and similarly it leads to a contradiction with the codimension one of $S$ as follows.
                             Since $\varphi_{(p,o)}$ for $p\in M$ calibrates the tangent cone $C$ at $(p,o) $ $-$ a hyperplane $\times$ hypercone Simons cone $C_{3,3}$,
                             let $\xi$ be the oriented unit top vector of the hyperplane 
                             and we get the calibration $\xi \lrcorner \varphi_{(p,o)}$ with constant coefficients calibrating $C_{3,3}$.
                             Note that as $C$ splits according to $\mathbb R^8\oplus \mathbb R^8$
                             and the factor $C_{3,3}$ completely sits inside the second $\R^8$.
                             So we can observe that the restriction of $\xi \lrcorner \varphi_{(p,o)}$ to  $\{0\}\times \R^8$,
                             again a calibration with constant coefficients but now in $\R^8$,
                              calibrates the hypercone $C_{3,3}$.
                             \footnote{
                             If we use $\mathcal P$ to stand for the inclusion 
                             $\{0\}\times \R^8\subset \R^8\times \R^8$,
                             then the restriction gives exactly $\mathcal P^*(\xi \lrcorner \varphi_{(p,o)})$.}
                             This is a contradiction.
                             
                             Apparently, by the above arguments  one can keep doing the Cartesian product among more copies of $(X,h)$ for duality obstruction in higher codimension.

                                  {\ }
                                  
                                  Next, we aim to drop the codimension requirement  in the above construction for the duality obstruction.
                                  For achieving this, we have to further modify the metrics.
                                  \begin{lem}\label{Mcontrol}
                                                               Suppose that
                                 $(X_1^{N_1+1}, h_1)$ and $(X_2^{N_2+1}, h_2)$ are two examples from \S \ref{62}.
                                  Let $M_i^{m_i}, S_i$ be calibrated by coflat calibration $\tilde \Phi_i$ with two singular points $p_{1,i}$ and $p_{2, i}$ (of $S_i$) in  $(X_i, h_i)$ for $i=1,2$.
                                  Then we can further conformally change $h_i$ to $\tilde h_i$
                                  such that 
                                  $\tilde \Phi_i$ is still a coflat calibration
                                  and that
                                  $\tilde \Phi_1\wedge \tilde \Phi_2$ is a smooth calibration away from
                                  $$
                                    {\mathcal S}\triangleq \big(X_1\vee_{(p_{1,1}, p_{1,2})} X_2\big)
                                    \cup
                                     \big(X_1\vee_{(p_{1,1}, p_{2,2})} X_2\big)
                                     \cup
                                      \big(X_1\vee_{(p_{2,1}, p_{1,2})} X_2\big)
                                      \cup 
                                       \big(X_1\vee_{(p_{2,1}, p_{2,2})} X_2\big)
$$
                        in $(X_1, \tilde h_1)\times (X_2, \tilde h_2)$.
                        Here  $X_1\vee_{(p_{i,1}, p_{j,2})} X_2$ means $\big(X_1\times \{p_{j,2}\}\big)\cup\big(\{p_{i,1}\}\times X_2\big)$.
                                 \end{lem}
                                   \begin{proof}
                                   Recall that we modify $\Phi_i$ to $\tilde \Phi_i$ in a small $\epsilon_i$-neighborhood $U_{\epsilon_i}(M_i)$ of $M_i$ in \cite{Z12}
                                   such that $\tilde \Phi_i$ is a simple form in $U_{3\epsilon_i/5}(M_i)$.
                                   So we can multiply  a conformal factor function being one in $U_{\epsilon_i/5}(M_i)\cup U_{\epsilon_i/5}(S_i)$ 
                                   and being a sufficiently large constant in $U_i^c\triangleq U^c_{2\epsilon_i/5}(M_i)\cap U^c_{2\epsilon_i/5}(S_i)$
                                   such that under the new metric $\tilde h_i$ we have the following pointwise
                                   $$
                                   \|\tilde \Phi_i\|^*_{\tilde h_i}\leq 
                                   C_i<1
                                   \, \, \, \,
                                   \text{ in }
                                   \,
                                   U_i^c
                                   $$
                                                                      for any wished $0<C_i<1$ (to be determined later) where $i=1,2$.
                                                                      
                                                                      Note that one can choose $\epsilon_i$ small enough for $\tilde \Phi_i$ to be simple in $U_{3\epsilon_i/5}(S_i)\sim \{p_{1,i}, p_{2,i}\}$ as well.
                                                                      The reason is that our local model is based on Theorem \ref{T1}
                                                                      which actually provides coflat calibration singular only at $p_{1,i}$ and $p_{2, i}$
                                                                      in some neighborhood $\Xi_i$ of $S_i$.
                                                                      In fact,  as  in \S \ref{spf} our local calibration $\varphi$ for a Lawlor cone is simple near  $p_{1,i}$ and $p_{2, i}$
                                                                      and the 
                                                                      gluing results in a simple $\Phi_i$ in $\Xi_i$ away from $p_{1,i}$ and $p_{2, i}$.
                                                                      Therefore, it is
                                                             easy to see that
                                                             $\tilde \Phi_1\wedge \tilde \Phi_2$
                                                             is a calibration in 
                                                             $                                    {\mathcal S}^c\cap \Big(\big(U_1\times X_2\big)\cup \big(X_1\times U_2\big)\Big)$
                                                             by the simpleness 
                                                             of $\tilde \Phi_i$ in 
                                                             $U_{3\epsilon_i/5}(M_i)\cup U_{3\epsilon_i/5}(S_i)$.
                                                            
                                                            So we only need to focus on the situation on $U^c_1\times U^c_2$.
                                                            Fix $q_i\in U^c_i$.
                                                            Let $Q$ be a unit $(m_1+m_2)$-simple vector of $\La^{m_1+m_2}T_{(q_1,q_2)}X_1\times X_2$.
                                                            Then immediately we have
                                                                  \begin{equation}\label{Qd}
                                                                  Q=v_1\wedge \cdots \wedge v_{m_1+m_2}=(a_1 e_1+b_1 f_1)\wedge \cdots \wedge (a_{m_1+m_2} e_{m_1+m_2}+b_{m_1+m_2} f_{m_1+m_2}).
                                                                  \end{equation}
                                   where $\{v_i\}$ is an orthonormal basis of span$Q$ and $v_i=a_i e_i+ b_i f_i$
                                  with  unit vectors $e_i\in T_{q_1}X_1$, $f_i\in T_{q_2}X_2$ and $a_i^2+b_i^2=1$.
                                  Therefore, in the full expansion of \eqref{Qd}
                                  there are at most
                                  $$
                                   \left( \begin{matrix} m_1+m_2\\ m_1\end{matrix} \right)-\text{many terms of form } 
                                   \pm a_I b_J e_I\wedge f_J
                                  $$
                                   where multi-indices $I=i_1\, i_2\, \cdots\, i_{m_1}$,
                                   $J=j_1\, j_2\, \cdots\, j_{m_2}$,
                                   coefficients $a_I=a_{i_1}
                                            \cdots a_{i_{m_1}}$,
                                   $b_J=b_{j_1}
                                              \cdots b_{j_{m_2}}$
                                   and $e_I\wedge f_J=e_{i_1}\wedge \cdots \wedge e_{i_{m_1}}\wedge f_{j_1}\wedge \cdots \wedge f_{j_{m_2}}$.

                                  As $|a_Ib_J|\leq 1$,
                                  it follows that
                                  the evaluation of $\tilde \Phi_1\wedge \tilde \Phi_2$ on $Q$
                                  is no larger than 
                                  $\left( \begin{smallmatrix} m_1+m_2\\ m_1\end{smallmatrix} \right)C_1C_2$.
                                  Now let us continue the argument in the first paragraph.
                                  If we choose $C_1=C_2=\frac{1}{(m_1+m_2)!}$,
                                  then  $\tilde \Phi_1\wedge \tilde \Phi_2$
                                  becomes a smooth calibration in 
                                   $U^c_1\times U^c_2$
                                   with respect to $\tilde h_1\oplus \tilde h_2$.
                                   Hence, together with the discussion in the second paragraph we finish the proof.
                                   \end{proof}

                                 If the degree of $\tilde \Phi_1\wedge \tilde \Phi_2$ is larger than the dimension of $\mathcal S$,
                                   then $\tilde \Phi_1\wedge \tilde \Phi_2$ in the above lemma is a coflat calibration in $(X_1, \tilde h_1)\times (X_2, \tilde h_2)$ 
                                   which 
                                   calibrates 
                    \begin{equation}\label{C4}
                             M_1\times M_2,\, \ \ \ M_1\times S_2,\, \ \  \ S_1\times M_2,\, \ \  \ S_1\times S_2.
                  \end{equation}
                                 Consequently, under $\tilde h_1\oplus \tilde h_2$,
                                 if 
                                 $ M_1\times M_2$ supports a globally defined smooth calibration on $X_1\times X_2$,
                                 then by the fundamental theorem of calibrated geometry it also calibrates $M_1\times S_2$.
                                 An argument similar to that before Lemma \ref{Mcontrol} 
                                 will lead to a contradiction with our starting local models by Theorem \ref{T2}.
                                 Hence, we encounter the duality obstruction.

                                   What can we say if the degree of $\tilde \Phi_1\wedge \tilde \Phi_2$ is less than the dimension of $\mathcal S$?
                                   Although  $\tilde \Phi_1\wedge \tilde \Phi_2$ is not a coflat calibration due to the large singular set,
                                   it can still serve as a calibration by combining the thoughts in \cite{Law} and \cite{z}.
                                   
                                   \begin{lem}\label{l64}
                                   With respect to the resulting metric $\tilde h_1\oplus \tilde h_2$ in Lemma \ref{Mcontrol},
                                   the currents in \eqref{C4} are homologically mass-minimzing in the same class.
                                   %
                                   \end{lem}
                                   
                                   \begin{proof}
                                   It is clear that those in \eqref{C4} are in the same homological class by our construction.
                                   Let us explain why they are homologically mass-minimizing.
                                   By Nash's embedding theorem \cite{Nash},
                                   we have
                                   isometric embedding
                                    $F_i: (X_i, \tilde h_i)\longrightarrow (\R^{\ell_i}, g_E)$
                                    and then one can pull back $\tilde \Phi_i$
                                    to a form $\bar \Phi_i=\pi_i^*(\tilde \Phi_i)$ in a small neighborhood of $F_i(X_i)$,
                                    where $\pi_i$ is the nearest projection to $F_i(X_i)$.
                                    Now we can perform mollifications to get a family of smooth forms $\big(\bar \Phi_i\big)_\epsilon$
                                    in a small neighborhood  of $F_i(X_i)$
                                    and restrict them back to get smooth closed forms $\varphi_i^\epsilon$ in $F_i(X_i)$.
                                   
                                    Mollification through singular set can be done 
                                    due to the following.
                                    By \S \ref{42} and \S \ref{spf}
                                    the resulting $\tilde \Phi_i=d\omega_i$
                                    around $o=p_{i,1}$ (or $p_{i,2}$)
                                    for 
                                    $\omega_i=h_2\psi_i$ (smooth and of bounded derivatives away from $o$ in local).
                                    Therefore, 
                                    \footnote{Cf. the proof of Corollary \ref{minCartProd} for more details.}
                                      \begin{equation}\label{dclosed}
\big(\bar \Phi_i\big)_\epsilon=                      
  \big(\pi_i^*d\omega_i\big)_\epsilon
=
             d\big(\pi_i^*\omega_i\big)_\epsilon
                                   \text{\ \ \ \ \ \ and \ \ \ \ \ \ }
                                    d\big(\bar \Phi_i\big)_\epsilon=0,
                                    \end{equation}
                              so  not hard to see
                              that                              the restriction $\varphi_i^\epsilon$ of
                                    smooth  $\big(\bar \Phi_i\big)_\epsilon$
                                       is also closed
                                     on the entire
                                     $F_i(X_i)$
                                     (identified as $X_i$ through the embedding).

                                 By \eqref{dclosed} we know that $d\varphi_i^\epsilon=0$
                                 and thus $d
                                 \big(
                                                     \varphi_1^\epsilon     \wedge 
                                                     \varphi_2^\epsilon  
                                                     \big)
                                                     =0$.
                                 So we have the equal integrals for terms in \eqref{C4}
                                  \begin{equation}\label{I4}
                            \int_{ M_1\times M_2}    \varphi_1^\epsilon \wedge \varphi_2^\epsilon
                            =
                             \int_{ M_1\times S_2}    \varphi_1^\epsilon \wedge \varphi_2^\epsilon
                             =
                               \int_{ S_1\times M_2}    \varphi_1^\epsilon \wedge \varphi_2^\epsilon
                               =
                                 \int_{ S_1\times S_2}    \varphi_1^\epsilon \wedge \varphi_2^\epsilon
                  \end{equation}
                               for all small $\epsilon>0$.
                    
                    According to the comass control method in the proof of Lemma \ref{Mcontrol},
                    we know that 
                    the comass $\|\varphi^\epsilon_1\wedge \varphi^\epsilon_2\|^*_{\tilde h_1\oplus \tilde h_2}<\frac{1}{(m_1+m_2)!}$ in $U^c_1\times U^c_2$
                    when $\epsilon$ is small
                    and the sequence $\{\varphi^\epsilon_1\wedge \varphi^\epsilon_2\}_{\epsilon\downarrow 0}$ 
                    serves as a calibration in the sense that, for any compactly supported competitor $T$ of normal current in $[M_1\times M_2]$,
     \begin{equation}\label{EM4}
                  \mathrm{\mathbf{M}}(M_1\times M_2)
                  =
                      \lim_{\epsilon\downarrow 0}\int_{ M_1\times M_2}    \varphi_1^\epsilon \wedge \varphi_2^\epsilon
                            =
                                     \lim_{\epsilon\downarrow 0}\int_{ T}    \varphi_1^\epsilon \wedge \varphi_2^\epsilon
                                     \leq 
                                     \mathrm{\mathbf{M}}(T)
               \end{equation}
                    as 
                    $$
                     \lim_{\epsilon\downarrow 0}
                     \max_{X_1\times X_2}\|\varphi^\epsilon_1\wedge \varphi^\epsilon_2\|^*_{\tilde h_1\oplus \tilde h_2}=1.
                     $$
                    Hence, we finish the proof.
                                   \end{proof}
                                  
                                  A  corollary of the above is  the minimality of the Cartesian product of two Lawlor cones.
                                  Here we provide a direct proof  in the Euclidean space.
                                  \begin{cor}\label{minCartProd}
                                  Let  $C_1\subset \R^{N_1+1}$ and $C_2\subset \R^{N_2+1}$ be two Lawlor cones.
                                  Then 
                                 the cone $C_1\times C_2\subset \R^{N_1+1}\oplus \R^{N_2+1}$ (not a regular cone) is also mass minimizing. 
                                  \end{cor}
                                  
                                  \begin{proof}
                                  Let $d\omega_i$ be the coflat calibration constructed  in \S \ref{42} for $C_i$.
                                  Note that  $\omega_i$ is singular (not defined) only at the origin, and smooth  away from the origin.
                                  Take $\sigma$ to be a mollifier, which is a smooth function supported in the unit ball of $\R^n$ centered at the origin with total integral one.
                                  Set $\sigma_\epsilon(z)=\frac{1}{\epsilon^n}\sigma(\epsilon z)$.
                                  Then mollification 
    \[
                                  f_\epsilon(x)=\int_{\R^n}f(y)\sigma_\epsilon (y-x) dy
\]
                                  is well defined for almost everywhere defined real-valued, locally bounded, measurable function $f$. 
                                  By mollifying $\omega_i$ we mean the mollification on its coefficient functions $f_{i, J}$ 
                                  (with respect to corresponding basis,
                                       i.e., 
                                       $\{dx_{j_1}\wedge \cdots \wedge dx_{j_{m_i}}\}$).
                                One nice property is that $f_\epsilon$ is smooth.
                                Another is 
                                        \begin{equation}\label{mollification}
                                      {\p_j} (f_\epsilon)
                                        =
                                        \int_{\R^n}f(y)
                                                  {\p_j} \big(\sigma_\epsilon (y-x) \big)dy
                                        =
                                          \int_{\R^n}
                                          (\p_jf)(y)
                                                 \sigma_\epsilon (y-x)  dy
                                                  =
                                        ({\p_j}f)_\epsilon
                                              \end{equation}
                                         whenever $f$ is a  smooth function.
                                         Actually 
                                         $
                                            {\p_j} (f_\epsilon)=
                                                  ({\p_j}f)_\epsilon
                                         $
                                         holds 
                                for Lipschitz $f$.
                                
                                Note that each $f_{i,J}$ is bounded, smooth away from the origin and may not continuous at the origin.
                                But this does not cause problem for 
                                 ${\p_j} (f_{i,J})_\epsilon=({\p_j} f_{i,J})_\epsilon$ 
                                 because 
                                 the integration by parts only needs to break into two pieces along the $x_j$-axis 
                                 and by the boundedness of $f_{i,J}$
                                 the contribution of integral over  the $x_j$-axis  can be ignored
                                 in \eqref{mollification}
                                 as $n>1$.
                                 Hence, each $\phi^\epsilon_i:=(d\omega_i)_\epsilon=d(\omega_i)_\epsilon$ is closed in $\R^{N_i+1}$
                                 and so is their product $\phi^\epsilon_1\wedge \phi^\epsilon_2$ in $\R^{N_1+1}\oplus \R^{N_2+1}$. 
                                 Besides,  easy to see  $\|\phi^\epsilon_i\|^*\leq 1$.
                                 
                                 As 
                                 argued 
                                 in \eqref{EM4}
                                 with
                                   $\lim_{\epsilon\downarrow 0}
                     \max
                     \|
                     \phi^\epsilon_1\wedge \phi^\epsilon_2
                     \|^*=1$,
                                 the family $\phi^\epsilon_1\wedge \phi^\epsilon_2$ 
can serve as a calibration and thus
$C_1\times C_2$ is mass-minimizing in $\R^{N_1+1}\oplus \R^{N_2+1}$.
                                  \end{proof}
                                  
                                  \begin{rem}
                                  An alternative way is to mollify $\phi=\phi_1\wedge \phi_2$ in $\R^{N_1+1}\oplus \R^{N_2+1}$.
                                  By the simpleness of $\phi_1$ and $\phi_2$, 
                                  clearly $\phi$ is a smooth simple calibration away from $\big(\{0\}\times \R^{N_2+1}\big)\cup \big(\R^{N_1+1}\times \{0\}\big)$.
                                 As argued in  the above proof each $\phi_\epsilon$ 
                                 is a smooth calibration with comass no larger than one. 
                                 Since the integral of  $\phi_\epsilon$ on the  truncated part of $C_1\times C_2$ within unit ball centered at the origin
                                 limits to its mass, 
                                 the conclusion of Corollary \ref{minCartProd}
                                 follows.
                                 Moreover,
                                 this result holds for Cartesian products of multiple Lawlor cones.
                                  \end{rem}
                                  
                                             \begin{rem}
                                  The proof of Corollary \ref{minCartProd}
                                  can work with  Lawlor's orginal non-continuous calibrations (away from the origin their potentials are Lipschitz forms).
                                  However, not like the study in \cite{z} where scope can be completely reduced to local,
                                  here in order to establish Lemma \ref{l64} 
                                  we need to suitably extend local calibration to a global one
                                  for our purpose
                                  and smooth calibration away from the singular set is particularly preferred.
                                  \end{rem}

                                  \begin{rem}
                                  Given two mass minimizing currents or cones,
                                  whether their Cartesian product is minimizing is still largely open.
                                  Special cases got verified in \cite{M}.
                                  \end{rem}

                         Now we are able to end up the paper with the following duality obstruction.         
                                  
                                  \begin{thm}
                                  With respect to the resulting metric $\tilde h_1\oplus \tilde h_2$ in Lemma \ref{Mcontrol},
$M_1\times M_2$ in $X_1\times X_2$ is homologically mass-minimizing but cannot support any globally defined smooth calibration.
                                  \end{thm}
                                   
                                   \begin{proof}
                                   The mass-minimality has been proved in the previous lemma.
                                  Let us assume that $M_1\times M_2$ supports a globally defined smooth calibration $\varpi$.
                                  Then by the fundamental theorem of calibrated geometry and Lemma \ref{l64},
                                  $\varpi$ also calibrates $M_1\times S_2$.
                                  
                                  We can repeat the argument before Lemma \ref{Mcontrol} for contradiction.
                                  Fix $q\in M_1$ and $p=p_{1}$ or $p_2$ one singular point of $S_2$
                                  and set $x=(q, p)$.
                                  Let $\xi$ be the oriented unit top vector of $T_qM_1$.
                                  Then
                                  $\xi\lrcorner \varpi_{x}$
                                  is a calibration form with constant coefficients in $\left(T_x (X_1\times X_2), (\tilde h_1\oplus \tilde h_2)_x\right)$.
                                  Since $\varpi$ calibrates $M_1\times S_2$,
                                  it can be seen that the restriction of $\xi\lrcorner \varpi_{x}$ to $\{0\}\times \R^{N_2+1}$
                                  calibrates the tangent cone of $S_2$ at $p$.
                                  This contradicts with our construction which is based on local model from Theorem \ref{T2}
                                  $-$ the tangent cone of $S_2$  at each singular point cannot support any calibration with constant coefficients.
                                  
                                  Therefore, in $(X_1\times X_2, \tilde h_1\oplus \tilde h_2)$ 
                                  the homologically mass-minimizing smooth submanifold $M_1\times M_2$
                                  supports no globally defined smooth calibrations at all.
                                   \end{proof}

                                   {\ }

{\ }

\end{document}